\documentclass{default}

\title{Distribution questions for trace functions with values in cyclotomic integers and their reductions}

\def\titleHead{Distribution of values of trace functions in cyclotomic integers}

\author{Corentin Perret-Gentil}
\address{ETH Zürich, Department of Mathematics}
\email{corentin.perretgentil@\{math.ethz.ch,gmail.com\}}

\date{October 2016. Revised June 2017.}


\begin{document}

\begin{abstract}
  We consider $\ell$-adic trace functions over finite fields taking values in cyclotomic integers, such as characters and exponential sums. Through ideas of Deligne and Katz, we explore probabilistic properties of the reductions modulo a prime ideal, exploiting especially the determination of their integral monodromy groups. In particular, this gives a generalization of a result of Lamzouri-Zaharescu on the distribution of short sums of the Legendre symbol reduced modulo an integer to all multiplicative characters and to hyper-Kloosterman sums.
\end{abstract}

\maketitle

\ifamsart
\tableofcontents
\fi

\section{Introduction}

The distribution of normalized exponential sums over finite fields in compact subsets of the complex numbers is an interesting question that has been studied by numerous authors such as Kummer, Hasse, Heath-Brown and Patterson (cubic Gauss sums), Deligne (Gauss sums), Katz (hyper-Kloosterman sums), or Duke, Friedlander and Iwaniec (Salié sums).

For example, Katz \cite{KatzGKM} showed that for $n\ge 2$, the normalized hyper-Kloosterman sums
  \begin{equation}
    \label{eq:KS}
    \Kl_{n,q}(x)=\frac{(-1)^{n-1}}{q^{(n-1)/2}}\sum_{\substack{x_1,\dots,x_n\in\F_q^\times\\ x_1\cdots x_n=x}}e\left(\frac{\tr(x_1+\dots+x_n)}{p}\right)\in\C \hspace{0.2cm} (x\in\F_q^\times),
  \end{equation}
  for $e(z)=e^{2i\pi z}$ ($z\in\C$), are equidistributed as $q\to+\infty$ in $\tr\SU_n(\C)$  if $n$ is odd (resp. $\tr\USp_n(\C)$ if $n$ is even), with respect to the pushforward of the Haar measure. This result builds on Deligne's construction of the Kloosterman sums as $\ell$-adic trace functions on $\P^1_{\F_q}$ \cite{Del2}, through Deligne's general equidistribution theorem and Katz's determination of monodromy groups.

\subsection{Exponential sums in cyclotomic fields}\label{subsec:ESCyclField}

Exponential sums are usually considered as complex numbers, but in general they actually take values in cyclotomic fields. For example, a multiplicative character of $\F_p$ of order $d$ has image in $\Q(\zeta_d)\subset\Q(\zeta_{p-1})$ for $\zeta_d\in\C^\times$ a primitive $d$th root of unity, while an additive character has image in $\Q(\zeta_p)$.

More generally, the functions we can form from additive and multiplicative characters of $\F_p$ by taking sums, products or convolutions (e.g. discrete Fourier transform), will take values in $\Q(\zeta_p)\Q(\zeta_{p-1})=\Q(\zeta_{p(p-1)})$.

Fisher \cite{Fisher95} extended Katz's vertical Sato-Tate law for (unnormalized) Kloosterman sums mentioned above to this perspective by studying their distribution as elements of $K=\Q(\zeta_p)$ via the Minkowski embedding $K\to\R^{p-1}$, with the hope of getting results on their distinctness. His equidistribution result with respect to a product of the Sato-Tate measure amounts to showing that it is possible to construct for every $\sigma\in\Gal(K/\Q)$ an $\ell$-adic trace function on $\P^1_{\F_q}$ corresponding to the $\sigma$-conjugate of the Kloosterman sum.

\subsection{Exponential sums in cyclotomic integers}\label{subsec:ESCyclInt}

A step further is to consider exponential sums, and functions $f:\F_q\to\C$ formed from them, as having values in cyclotomic integers, say $\Oc=\Z[\zeta_d]$ for some $d\ge 1$. This holds true for characters and this property is again stable by the operations we mentioned above.

Wan \cite{Wan95} adopted such a point of view and studied the minimal polynomial of Kloosterman sums, improving some of Fisher's results.

Up to localizing\footnote{For $\alpha\in\Oc$ (resp. a prime ideal $\lf\normal\Oc$), we denote by $\Oc_\alpha$ (resp. $\Oc_\lf$) the localization at $\alpha$ (resp. at $\lf$).}, we can also consider normalizations: indeed, by the evaluation of quadratic Gauss sums, $\sqrt{p}\in\Z[\zeta_{4p}]$, so for example $\Kl_{n,q}(\F_q^\times)\subset \Z[\zeta_{4p}]_{q^{(n-1)/2}}$.

\subsection{Reductions of exponential sums in residue fields}\label{subsec:ESCyclIntRed}

For any nonzero prime ideal $\lf\normal\Oc$ (possibly restricted to be above a large enough prime to handle normalizations), we can then study the reductions modulo $\lf$ of exponential sums and related functions $f:\F_q\to\Oc$ in the corresponding residue field $\F_\lf=\Oc/\lf\cong\Oc_\lf/\lf\Oc_\lf$.

In the case of Kloosterman sums, if $\lf$ is a prime ideal of $\Oc=\Z[\zeta_{4p}]$ above an odd prime $\ell\neq p$, then we have the reduction
\[\Kl_{n,q}: \F_q^\times\to\Oc_\lf\to\F_\lf.\]

Distribution questions concerning the values in $\F_\lf$ as $q\to+\infty$ can then be examined. For example, Lamzouri and Zaharescu \cite{LamzModm} studied the distribution of short sums of the Legendre symbol $\chi_p: \F_p\to\{\pm 1\}$ reduced modulo an integer $\ell\ge 2$. Specifically, they show that
\[\frac{|\{1\le k\le p : \sum_{x=1}^k \chi_p(x)\equiv a\pmod{\ell}\}|}{p}=\frac{1}{\ell}+ O \left(\left(\frac{\ell}{\log{p}}\right)^{\frac{1}{2}}\right)\]
uniformly with respect to $a\in\Z/\ell$. A probabilistic model with sums of independent random variables uniformly distributed in $\{\pm 1\}$ is used; its accuracy is proved through a bound derived from the Riemann hypothesis for curves over finite fields.

\subsection{Trace functions over finite fields}

Additive and multiplicative characters of finite fields, hyper-Kloosterman sums and more general exponential sums are particular examples of trace functions of constructible middle-extension sheaves of $\overline\Q_\ell$-modules on $\P^1_{\F_q}$, as they appear in particular in the works of Katz (see for example \cite{KatzGKM} and \cite{KatzESDE}), and more recently in the series of papers by Fouvry, Kowalski, Michel and others (see \cite{FKMPisa}, \cite[Section 6]{Polymath8a} or \cite{PG16} for surveys).

Herein, we will mainly consider:
\begin{itemize}
\item Multiplicative characters $\chi: \F_q\to\Z[\zeta_d]$ of order $d$, eventually composed with a rational polynomial, realized as trace functions of \textit{Kummer sheaves}, with $\chi(0)=\chi(\infty)=0$.
\item Hyper-Kloosterman sums \eqref{eq:KS} $\Kl_{n,q}: \F_q\to \Z[\zeta_{4p}]_{q^{(1-n)/2}}$, realized as trace functions of \textit{Kloosterman sheaves}, with $\Kl_{n,q}(0)=(-\sqrt{q})^{n-1}$.
\item Functions $f: \F_q\to\N$ counting points on families of curves on $\P^1_{\F_q}$ parametrized by an open of $\P^1$. For example, for $f\in\F_q[X]$ a fixed squarefree polynomial of degree $2g\ge 2$, we have the family of hyperelliptic curves given by the affine models
  \[X_z : y^2=f(x)(x-z) \ (z\in\F_q),\]
  as constructed by Katz-Sarnak \cite[Chapter 10]{KatzSarnak91}.
\end{itemize}

The observations on the images of these functions from Sections \ref{subsec:ESCyclField}--\ref{subsec:ESCyclIntRed} happen to translate on the level of sheaves: they are actually sheaves of $\Oc_\lambda\le\overline\Q_\ell$-modules, where $\Oc$ is the ring of integers of a cyclotomic field and $\lambda$ is an $\ell$-adic valuation corresponding to a prime ideal $\lf\normal\Oc$ above an auxiliary prime $\ell\neq p$.

For the second example, this follows from the fact that the $\ell$-adic Fourier transform is defined on the level of $\Oc_\lambda$-modules (see \cite[Chapter 5]{KatzGKM}).

\subsubsection{Reductions}

The reduction of the trace function modulo $\lf$ then corresponds to the trace function of the reduced sheaf of $\F_\lf$-modules, for $\F_\lf\cong\Oc_\lambda/\lf\Oc_\lambda$ the residue field. This implies that we can use the $\ell$-adic formalism and the ideas of Deligne and Katz to study distribution questions of these reduced trace functions.

As in \cite{LamzShortSums} (that we generalized in \cite{PGGaussDistr16}) and \cite{LamzModm}, a key idea is to use a probabilistic model, based on Deligne's equidistribution theorem.

\subsubsection{Monodromy groups}

For Kloosterman sums, an important input is the determination of the $\Oc_\lambda$-integral monodromy groups of Kloosterman sheaves \cite{PGIntMonKS16}, analogous to the determination of the monodromy groups over $\overline\Q_\ell$ by Katz \cite{KatzGKM}, when $\ell$ is large enough depending only on the rank.

This was already known by results of Gabber, Larsen and Nori, but for $\ell$ large enough depending on $q$ and with an ineffective constant, which would have been unusable for our applications.

\subsubsection{Other examples}

The above setup also applies to other $\ell$-adic trace functions such as hypergeometric sums (as defined by Katz in \cite[Chapter 8]{KatzESDE}) and general exponential sums of the form
\[\frac{-1}{\sqrt{q}}\sum_{y\in\F_q} e\left(\frac{\tr(xf(y)+h(y))}{p}\right)\chi(g(y)) \hspace{0.2cm} (x\in\F_q),\]
for $f,g,h\in\Q(X)$ rational functions and $\chi$ a multiplicative character on $\F_q^\times$.

However, the determination of their $\Oc_\lambda$-integral monodromy groups is the object of future work. If we showed that these are as large as possible (hence classical groups), as is the case for the monodromy groups over $\overline\Q_\ell$ by Katz's work \cite{KatzESDE}, the results below would hold as well.

\subsection{Overview of the results}

Given an abelian group $A$, a function $f: \F_q\to A$ and a subset $E\subset\F_q$, we denote by
\[S(f,E)=\sum_{x\in E}f(x)\]
the partial sum over $E$. For $x\in \F_q$, we let $E+x=\{e+x : e\in E\}$ be the translate of $E$ by $x$. With the uniform measure on $\F_q$, we can consider the $A$-valued random variable $\left(S(f,E+x)\right)_{x\in\F_q.}$

\subsubsection{Equidistribution for shifted sums}\label{subsubsec:introEDSSS}

The first results concern the distribution of short shifted sums, and are analogues of the questions answered in \cite{LamzShortSums} and \cite{PGGaussDistr16} (where the random variables were shown to be gaussian under some ranges, generalizing a result of Erd\H{o}s-Davenport).
\begin{proposition}[Kloosterman sums]\label{prop:EDShiftsKl}
  For $n\ge 2$, and $\lf\normal\Z[\zeta_{4p}]$ a prime ideal above a prime $\ell\gg_n 1$ distinct from $p$ with $\ell\equiv 1\pmod{4}$, let $\Kl_{n,q}: \F_q\to\F_\lf$ be the reduction modulo $\lf$ of the (normalized) Kloosterman sum on $\P^1_{\F_q}$. For any $I\subset\F_q$ of size $L$, the probability
  \[P\Big(S(\Kl_{n,q},I+x) \equiv a\Big)\]
  is given by
  \[\frac{1}{|\F_\lf|}+
    \begin{cases}
      O_n \left(|\F_\lf|^{-L\frac{n^2-1}{2}}+|\F_\lf|^{L \frac{n^2+n-2}{2}+n(n-1)-1}q^{-\frac{1}{2}}\right)&\text{if } n\text{ odd}\\
      O_n \left(|\F_\lf|^{-L\frac{n(n+2)}{8}}+|\F_\lf|^{L \frac{n(n+2)}{4}+\frac{n^2-2}{2}}q^{-\frac{1}{2}}\right)&\text{if } n\text{ even}
  \end{cases}
  \]
  uniformly for all $a\in\F_\lf$. This also holds for unnormalized Kloosterman sums.
\end{proposition}
\begin{remark}
  When $p\equiv 1\pmod{4}$ or $n$ is odd, one may replace $\zeta_{4p}$ by $\zeta_p$ and the result holds without restriction on $\ell\pmod{4}$. The same remark will apply to the subsequent statements involving Kloosterman sums.
\end{remark}

A similar result is valid for point-counting functions on families of hyperelliptic curves.

\begin{proposition}[Multiplicative characters]\label{prop:EDShiftsChi}
  We let:
  \begin{enumerate}[leftmargin=*]
  \item $d\ge 2$ be an integer, $\lf\normal\Oc=\Z[\zeta_d]$ be a nonzero prime ideal,
  \[\chi: \F_q\to\F_\lf\]
  be the reduction modulo $\lf$ in $\F_\lf=\Oc/\lf$ of a multiplicative character of order $d$;
\item $f=f_1/f_2\in\Q(X)$ whose poles and zeros have order not divisible by $d$;
\item $\delta\in(0,1)$ be such that\footnote{If $\F_\lf=\F_\ell$, or if $[\F_\lf:\F_\ell]$ is prime, or if $\delta>1/2$, this condition is simply $d\ge|\F_\lf|^\delta$.} $d>|\F_\lf|^\delta$ and $d/(d,|F^\times|)\ge|\F_\lf|^{\delta}$ for every proper subfield $F\le\F_\lf$;
\item $I\subset\F_q$ of size $L$; if $f\neq X$, we assume moreover that $|I|=1$ or $I\subset[1,p/\deg(f_1))^e$ with respect to an arbitrary $\F_p$-basis of $\F_q$, identifying the latter with $\{1,\dots,p\}^e$.
  \end{enumerate}
  Then there exists $\alpha=\alpha(\delta)>0$ such that
  \begin{equation*}
    P\Big(S(\chi\circ f,I+x) \equiv a\Big)=\frac{1}{|\F_\lf|}+O_f\left(\frac{1}{|\F_\lf|^{L\alpha}}+\frac{Ld^{L+1}}{q^{1/2}|\F_\lf|^{\min(L\alpha,1)}}\right)
  \end{equation*}
  uniformly for all $a\in\F_\lf$. Moreover, if $\delta>1/2$, we can choose $\alpha(\delta)=\delta-1/2$; if $\F_\lf=\F_\ell$ with $\delta>1/3$, we can choose
      \begin{equation}
      \label{eq:explicitalpha}
      \alpha(\delta)=\begin{cases}
        \frac{3\delta-1}{8}&\text{if }\delta\in(1/3,1/2]\\
        \frac{5\delta-2}{8}&\text{if }\delta\in(1/2,2/3]\\
        \delta-\frac{2}{3}&\text{if }\delta\in(2/3,1].
      \end{cases}  
    \end{equation}
\end{proposition}

The ranges of the various parameters will be studied in due time. Propositions \ref{prop:EDShiftsKl} and \ref{prop:EDShiftsChi} will be particular cases of Theorem \ref{thm:equidistribution} below.

\subsubsection{Generalizations of \cite{LamzModm} to trace functions: distribution of families of short sums}\label{subsubsec:introDistrFam}

Next, we generalize a result of Lamzouri-Zaharescu \cite{LamzModm} to the distribution of various families of sums of reduced trace functions, in particular multiplicative characters of any order and Kloosterman sums.

\begin{proposition}[Shifts of small subsets]\label{prop:shiftsSmallSubsets}
  Let $\varepsilon\in(0,1/4)$ and let $t:\F_q\to\F_\lf$ be either $t=\Kl_{n,q}$ as in Proposition \ref{prop:EDShiftsKl} or $t=\chi\circ f$ as in Proposition \ref{prop:EDShiftsChi}. Let $E\subset\F_q$ be a ``small'' subset. Then
  \[P\Big(S(t,E+x)\equiv a\Big)=\frac{1}{|\F_\lf|}+
  \begin{cases}
    O_{\varepsilon,n} \left(\frac{1}{q^{1/4-\varepsilon}}+ \left(\frac{|E|\log{|\F_\lf|}}{\log{q}}\right)^{\frac{1}{2}}\right)\\
    O_{\varepsilon,f} \left(\frac{1}{q^{1/4-\varepsilon}}+ \left(\frac{|E|\log{d}}{\log{q}}\right)^{\frac{1}{2}}\right)
  \end{cases}\]
  for Kloosterman sums, respectively multiplicative characters, uniformly for all $a\in\F_\lf$.
\end{proposition}
(This will be a particular case of Proposition \ref{prop:EDfixedSubsets}, where the conditions on $E$ will be made precise).\\

The second example generalizes the result of \cite{LamzModm} to all multiplicative characters:
\begin{proposition}[Partial intervals]\label{prop:partialIntervals1}
  Let $\varepsilon\in(0,1/4)$ and let $t=\chi\circ f: \F_p\to\F_\lf$ be as in Proposition \ref{prop:shiftsSmallSubsets}. Then
  \[\frac{|\{1\le k\le p : S(\chi\circ f,\{1,\dots,k\})\equiv a\}|}{p}\]
  is equal to
  \[\frac{1}{|\F_\lf|}+O_{\varepsilon,f} \left(\frac{1}{p^{1/4-\varepsilon}}+\left(\frac{\log{d}}{\log{p}}\right)^{\frac{1}{2}}+\delta_{f\neq X} \left(\frac{|\F_\lf|\log{p}}{p\log{d}}\right)^{\frac{1}{2}}\right)\]
  uniformly for all $a\in\F_\lf$.
\end{proposition}
(This will be a particular case of Proposition \ref{prop:partialIntervals}).\\

The method does not allow this to be generalized to Kloosterman sums, but we can nonetheless do the following:
\begin{proposition}[Partial intervals with shifts of small subsets]\label{prop:partialIntervalShifts1}
  We consider the situation of Proposition \ref{prop:shiftsSmallSubsets} with a fixed choice of a $\F_p$-basis of $\F_q$ giving an identification $\F_q\cong\F_p^e=\{1,\dots,p\}^e$. We let $E_2,\dots,E_e\subset\F_p$ be ``small'' subsets. Then the density
  \[\frac{|\{(x_1,\dots,x_e)\in\{1,\dots, p\}^e : S(t,\{1,\dots, x_1\}\times\prod_{i=2}^e (E_i+x_i))\equiv a\}|}{q}\]
  is equal to
  \[\frac{1}{|\F_\lf|}+
  \begin{cases}
    O_{\varepsilon,n} \left(\frac{1}{q^{1/4-\varepsilon}}+ \left(\frac{\log{|\F_\lf|}\prod_{i=2}^e |E_i|}{\log{q}}\right)^{\frac{1}{2}}\right)\\
    O_{\varepsilon,f} \left(\frac{1}{q^{1/4-\varepsilon}}+ \left(\frac{\log{d}\prod_{i=2}^e |E_i|}{\log{q}}\right)^{\frac{1}{2}}\right)&
  \end{cases}
\]
for Kloosterman sums, respectively multiplicative characters, uniformly for all $a\in\F_\lf$.
\end{proposition}
(This will be a particular case of Proposition \ref{prop:partialIntervalShifts}, where the condition on $E_i$ will be made precise).\\

Again, these examples also apply to unnormalized Kloosterman sums and functions counting points on families of hyperelliptic curves (normalized or not).


\subsection{Structure of the paper}
This article is structured as follows:
\begin{itemize}
\item In Section \ref{sec:setup}, we setup the technical framework we will work in to handle reductions of $\ell$-adic trace functions over finite fields, and we define precisely the examples we will consider.
\item In Section \ref{sec:model}, we defined a probabilistic model for short sums of $\ell$-adic trace functions, inspired by Deligne's equidistribution theorem.
\item In Section \ref{sec:proofThmModel}, we prove that this model is accurate (akin to what is done in \cite{PGGaussDistr16} for sheaves of $\overline\Q_\ell$-modules).
\item In Section \ref{sec:compModel}, we make preliminary computations and observations in the model, in particular regarding ``Gaussian sums'' in monodromy groups.
\item In Sections \ref{sec:EDSSS} and \ref{sec:DistrFam}, we finally prove the results introduced in Sections \ref{subsubsec:introEDSSS} and \ref{subsubsec:introDistrFam}, respectively.
\end{itemize}

\ifamsart
\begin{acknowledgements}
  The author would like to thank his supervisor Emmanuel Kowalski for guidance and advice during this project, as well as the anonymous referees for careful readings and valuable comments. It is a pleasure to acknowledge in particular the influence of the works of \'Etienne Fouvry, Nicholas Katz, Emmanuel Kowalski, Youness Lamzouri, Philippe Michel and Alexandru Zaharescu. This work was partially supported by DFG-SNF lead agency program grant 200021L\_153647. The final corrections were made while the author was in residence at the Mathematical Sciences Research Institute in Berkeley, California, during the Spring 2017 semester, supported by the National Science Foundation under Grant No. DMS-1440140. The results also appear in the author's PhD thesis \cite{PG16}.
\end{acknowledgements}
\fi
\section{Technical setup and examples}\label{sec:setup}

Let $\F_q$ be a finite field of odd characteristic $p$. For an integer $d\ge 2$, let $E=\Q(\zeta_d)$ be the $d$th cyclotomic field with ring of integers $\Oc$.  We fix an auxiliary prime $\ell\neq p$ and a prime ideal $\lf\normal\Oc$ above $\ell$, corresponding to a valuation $\lambda$ of $E$ extending the $\ell$-adic valuation on $\Q$. Let $E_\lambda$ and $\Oc_\lambda$ be the completions, and let
\[\pi: \Oc_\lambda\to\F_\lf\]
be the reduction map in the residue field $\F_\lf=\Oc/\lf\cong\Oc_\lambda/\lf\Oc_\lambda$.

\subsection{Review of $\ell$-adic sheaves on $\P^1_{\F_q}$}

In the following, let $A$ be either $\overline\Q_\ell$, $E_\lambda$, $\Oc_\lambda$ or $\F_\lf$.

\subsubsection{Definitions and basic properties}

As in \cite{KatzGKM}, we consider a constructible sheaf $\Fc$ of $A$-modules on $\P^1/\F_q$ which is middle-extension, i.e. for every nonempty open $j: U\to\P^1$ on which $j^*\Fc$ is lisse, we have $\Fc\cong j_*j^*\Fc$. For simplicity, we shall from now on simply call $\Fc$ an ``\textit{$\ell$-adic sheaf of $A$-modules on $\P^1_{\F_q}$}''.

We write $\Sing(\Fc)=X(\overline\F_q)\backslash U_\Fc(\overline\F_q)$ for the set of \textit{singularities} of $\Fc$, where $U_\Fc$ is the maximal open set of lissity of $\Fc$.\\

We recall that the category of $\ell$-adic sheaves of $A$-modules of generic rank $n$ on $\P^1_{\F_q}$ is equivalent to the category of continuous $\ell$-adic Galois representations
\[\rho_\Fc: \pi_{1,q}\to\GL_n(A),\]
for $\pi_{1,q}=\Gal\left(\F_q(T)^\sep/\F_q(T)\right)$ the étale fundamental group (see \cite[Theorem 7.13]{KiRu14}) and $A=\Fc_{\overline\eta}$ for $\overline\eta$ the geometric generic point of $\P^1_{\F_q}$ corresponding to the chosen separable closure.

We say that $\Fc$ is \textit{irreducible} (resp. \textit{geometrically irreducible}) if $\rho_\Fc$ (resp. the restriction of $\rho_\Fc$ to $\pi_{1,q}^\geom:=\Gal(\F_q(T)^\sep/\overline\F_q(T))\normal\pi_{1,q}$) is an irreducible representation.

For $x\in\P^1(\F_q)$, we denote by
\begin{itemize}
\item $I_x\normal D_x\le \pi_{1,q}$ \textit{inertia (resp. decomposition) groups} at (the valuation corresponding to) $x$.
\item $\Frob_{x,q}\in D_x/I_x\cong\Gal(\overline\F_q/\F_q)$ an element mapping to the geometric Frobenius $\Frob_q$.
\item $\Fc_{\overline\eta}^{I_x}$ the space of invariants of $\Fc_{\overline\eta}$ with respect to the action of $I_x$. Note that $\rho_\Fc(\Frob_{x,q})\in\GL(\Fc_{\overline\eta}^{I_x})$ is well-defined.
\item $\Swan_x(\Fc)\in\Z_{\ge 0}$ the \textit{Swan conductor} of $\Fc$ at $x$.
\end{itemize}
These are defined up to conjugation. As in the works of Fouvry-Kowalski-Michel (see e.g. \cite{AlgebraicTwists}), we consider the \textit{conductor}
\[\cond(\Fc)=\rank(\Fc)+|\Sing(\Fc)|+\sum_{x\in\Sing(\Fc)} \Swan_x(\Fc)\]
of $\Fc$, which combines three invariants of the sheaf to measure its ``complexity'' (with respect to dimension and ramification).\\

The \textit{trace function} of $\Fc$ is the map 
\begin{align*}
  t_\Fc:\hspace{0.4cm} &\P^1(\F_q)\to A\\
                       &x\mapsto \tr\left(\rho_\Fc(\Frob_{x,q})\mid \Fc_{\overline\eta}^{I_x}\right).
\end{align*}

If $A$ has characteristic zero, we say that $\Fc$ is \textit{pointwise pure of weight $0$} if for every finite extension $\F_{q'}/\F_q$ and every $x\in U_\Fc(\F_{q'})$, the eigenvalues of $\rho_\Fc(\Frob_{x,q'})$ are Weil numbers $\alpha$ of weight $0$, i.e. $\alpha\in\overline\Q$ and for any embedding $\iota: \overline\Q\to\C$, we have $|\iota(\alpha)|=1$. In this case,
\[||t_\Fc||_{\iota,\infty}:=\max_{x\in\P^1(\F_q)}|\iota(t_\Fc(x))|\le\rank(\Fc)\le\cond(\Fc),\]
for any such $\iota$, which is clear at lisse points and a result of Deligne \cite[(1.8.9)]{Del2} at singularities.

For $\Fc,\Gc$ two sheaves of $A$-modules on $\P^1_{\F_q}$, we denote by:
\begin{itemize}
\item $\Fc\otimes\Gc$ the ``middle tensor product'', i.e. the sheaf of $A$-modules on $\P^1_{\F_q}$ corresponding to the representation $\rho_\Fc\otimes\rho_\Gc$
\item $D(\Fc)$ the ``dual sheaf'', i.e. the sheaf of $A$-modules on $\P^1_{\F_q}$ corresponding to the dual/contragredient representation of $\rho_\Fc$.
\end{itemize}

\subsubsection{Sums of trace functions} Deligne's analogue of the Riemann hypothesis over finite field for weights of étale cohomology groups \cite{Del2} along with the Grothendieck-Lefschetz trace formula and the Euler-Poincaré formula of Grothendieck-Ogg-Safarevich gives the following asymptotic estimate for sums of trace functions:
\begin{theorem}\label{thm:sumTF}
  If $A$ as above has characteristic zero and if $\Fc$ is a sheaf of $A$-modules on $\P^1_{\F_q}$ which is pointwise pure of weight $0$, we have
  \[\sum_{x\in U_\Fc(\F_q)} t_\Fc(x)=q\cdot\tr\left(\Frob_q\,|\,\Fc_{\pi_{1,q}^\geom}\right)+O \left(E(\Fc)\sqrt{q}\right)\]
  with respect to an arbitrary embedding $\iota: \overline\Q\to\C$, where $\Fc_{\pi_{1,q}^\geom}$ is the space of coinvariants of the representation $\rho_\Fc$ restricted to $\pi_{1,q}^\geom$, with the action of $\Frob_q\in\Gal(\overline\F_q/\F_q)=\pi_{1,q}/\pi_{1,q}^\geom$, and
  \[E(\Fc)= \rank(\Fc)\left[|\Sing(\Fc)|-1+\sum_{x\in\Sing(\Fc)}\Swan_{x}(\Fc)\right]\ll\cond(\Fc)^2.\]
  Moreover, the same relation holds for $\sum_{x\in\F_q} t_{\Fc}(x)$.
\end{theorem}
\begin{proof}
  See for example \cite[Exposé 6]{DelEC}, \cite[Chapter 2]{KatzGKM}, \cite[Section 9]{AlgebraicTwists} or \cite[Section 2.2]{PG16}.
\end{proof}

It follows that geometrically irreducible $\ell$-adic trace functions on $\P^1_{\F_q}$ are ``almost orthogonal'':

\begin{corollary}\label{cor:sumTF}
  If $A$ as above has characteristic zero and if $\Fc$, $\Gc$ are geometrically irreducible sheaves of $A$-modules on $\P^1_{\F_q}$ which are pointwise pure of weight $0$, then
  \[\sum_{x\in\F_q} t_\Fc(x)\overline{t_\Gc(x)}=C(\Fc,\Gc)q+O(\cond(\Fc)^2\cond(\Gc)^2\sqrt{q}),\]
  with $C(\Fc,\Fc)=1$ and $C(\Fc,\Gc)=0$ if $\Fc$ and $\Gc$ are not geometrically isomorphic, the conjugate being interpreted with respect to any embedding $\iota: \overline\Q\to\C$.
\end{corollary}
\begin{proof}
  The sheaf $\Hc=\Fc\otimes D(\Gc)$ satisfies $t_\Fc(x)\overline{t_\Gc(x)}=t_{\Hc}(x)$ when $x\not\in\Sing(\Fc)\cup\Sing(\Gc)$, and $\cond(\Hc)\ll\cond(\Fc)^2\cond(\Gc)^2$ (see the references above). Hence
  \begin{eqnarray*}
    \sum_{x\in\F_q} t_\Fc(x)\overline{t_\Gc(x)}&=&\sum_{x\in U_\Hc(\F_q)} t_\Gc(x)+O\left(||t_{\Fc}t_{\Gc}||_{\iota,\infty}(|\Sing(\Fc)|+|\Sing(\Gc)|)\right)\\
                                               &=&\sum_{x\in \F_q} t_\Gc(x)+O\left(\cond(\Fc)^2\cond(\Gc)^2\right)\\
                                               &=&q\cdot\tr\left(\Frob_q\,|\,\Hc_{\pi_{1,q}^\geom}\right)+O\left(\cond(\Fc)^2\cond(\Gc)^2\sqrt{q}\right),
  \end{eqnarray*}
  where the last equality is Theorem \ref{thm:sumTF}. By Schur's Lemma, $\dim\Hc_{\pi_{1,q}^\geom}=1$ if $\Fc$ and $\Gc$ are geometrically isomorphic, and is zero otherwise. If $\Fc=\Gc$, the action of the Frobenius on the 1-dimensional vector space $\Hc_{\pi_{1,q}^\geom}$ is moreover trivial.
\end{proof}

\subsection{Reductions}\label{subsec:reductions}

Let $\Fc$ be an $\ell$-adic sheaf of $\Oc_\lambda$-modules on $\P^1_{\F_q}$, corresponding to a representation
\[\rho_\Fc: \pi_{1,q}\to\GL_n(\Oc_\lambda).\]
By reduction modulo $\lf$, we obtain an $\ell$-adic sheaf of $\F_\lf$-modules, corresponding to the representation $\pi\circ\rho_\Fc$, with trace function $\pi\circ t_\Fc$.
\begin{equation*}
  \xymatrix@C=1.8cm@R=0.2cm{
    &\GL_n(\Oc_\lambda)\ar[r]^{\hspace{0.4cm}\tr}\ar[dd]^\pi&\Oc_\lambda\ar[dd]^\pi&\\
    \pi_{1,q}\ar[ur]^{\rho_\Fc}\ar[dr]&&&\ar[ul]_{t_\Fc}\ar[dl]\F_q\\
    &\GL_n(\F_\lf)\ar[r]^{\hspace{0.4cm}\tr}&\F_\lf&
  }
\end{equation*}

\begin{remark}\label{rem:splittingCycl}
  By the theory of ramification in cyclotomic fields, we have $|\F_\lf|=\ell^m$ with $m$ the multiplicative order of $\ell$ modulo $d$ (see \cite[Theorem 2.13]{Was97}). In particular,
  \[|\F_\lf|\equiv 1\pmod{d}\text{ and } d<|\F_\lf|.\]
  Moreover, $\F_\lf=\F_\ell$ (i.e. $\ell$ splits completely) if and only if $\ell\equiv 1\pmod{d}$.
\end{remark}

\begin{remark}\label{rem:tImageO}
  In practice, $t_\Fc: \F_q\to\Oc_\lambda$ will actually have image in $\Oc$ or $\Oc_\alpha$ for some $\alpha\in\Oc\backslash\lf$ when we normalize (see the examples given in the introduction). We recall that our motivation is to study the reduction of a function $\F_q\to\Oc$ modulo almost any prime ideal of $\Oc$. The following diagram summarizes the different rings considered and natural maps between them:
  \begin{equation*}
    \xymatrix@R=0.5cm{
      &&E\ar[r]&E_\lambda&\\
      \Oc\ar[r]&\Oc_\alpha\ar[r]&\Oc_\lf\ar[u]\ar[r]&\Oc_\lambda\ar[r]^{\text{mod }\lf}\ar[u]&\F_\lf.
    }
  \end{equation*}
  
\end{remark}

\subsection{Examples}

\subsubsection{Kummer sheaves}

\begin{proposition}\label{prop:Kummer}
  Let $\Oc=\Z[\zeta_d]$ and let $\chi: \F_q^\times\to\Oc^\times$ be a nontrivial multiplicative character of order $d\ge 2$, $\lambda$ be an $\ell$-adic valuation on $\Oc$ corresponding to a prime ideal $\lf$ above $\ell$, and $f=f_1/f_2\in\F_q(X)$ which is not a $d$-power. We assume that $f$ has no zero or pole of order divisible by $d$. There exists a sheaf $\Lc_{\chi(f)}=\Lc_{\chi(f),\lambda}$ of $\Oc_\lambda$-modules on $\F_q$ with:
  \begin{enumerate}
  \item trace function $\chi\circ f$ (under the convention that $\chi(0)=\chi(\infty)=0$).
  \item tame singularities at the zeros and poles of $f$.
  \item $\cond(\Lc_{\chi(f)})\le 1+\deg(f_1)+\deg(f_2)$.
  \end{enumerate}
  By reduction modulo $\lf$, this gives a sheaf of $\F_\lf$-modules with the same properties and trace function $\chi\circ f\pmod{\lf}$, for $\F_\lf=\Oc_\lambda/\lf\Oc_\lambda$.
\end{proposition}
\begin{proof}
  See e.g. \cite[Section 4.3]{KatzGKM}.
\end{proof}

\begin{notation}
  For $f=f_1/f_2\in\F_q(X)$ with $f_1,f_2\in\F_q[X]$ coprime, we write $\deg(f)=\max(\deg(f_1),\deg(f_2))$, so that $\cond(\Lc_{\chi(f)})\ll\deg(f)$.
\end{notation}

\subsubsection{Kloosterman sheaves}
\begin{proposition}\label{prop:Klcn}
  Let $n\ge 2$ be an integer and $\lambda$ be an $\ell$-adic valuation on $\Oc=\Z[\zeta_{4p}]$ corresponding to a prime ideal $\lf$ above $\ell$. There exists a Kloosterman sheaf $\Klc_n=\Klc_{n,\lambda}$ of $\Oc_\lambda$-modules on $\P^1_{\F_q}$, of rank $n$ and with trace function corresponding to the Kloosterman sum
  \[x\mapsto\Kl_{n,q}(x)=\frac{(-1)^{n-1}}{q^{\frac{n-1}{2}}}\sum_{\substack{x_1,\dots,x_n\in\F_q^\times\\ x_1\dots x_n=x}}e\left(\frac{\tr(x_1+\dots+x_n)}{p}\right) \ (x\in\F_q^\times),\]
  and $\Kl_{n,q}(0)=(-\sqrt{q})^{n-1}$. Moreover, $\Klc_n$ is geometrically irreducible, lisse on $\G_m$, $\Swan_\infty(\Klc_n)=1$, $\Swan_0(\Klc_n)=0$, $\cond(\Klc_n)=n+3$, and we note that \[\Kl_{n,q}(x)\in\Oc_{q^{(n-1)/2}}\le\Oc_\lf\] for all $x\in\F_q^\times$.

  By reduction modulo $\lf$, this gives a sheaf of $\F_\lf$-modules on $\P^1_{\F_q}$ with the same properties and trace function equal to $\Kl_{n,q}\pmod{\lf}$, for $\F_\lf=\Oc_\lambda/\lf\Oc_\lambda$.

If $p\equiv 1\pmod{4}$ or $n$ is odd, we may replace $\Oc$ by $\Z[\zeta_p]$.
\end{proposition}
\begin{proof}
  Recall that
  \[\varepsilon_p\sqrt{p}\in\Z[\zeta_p]\text{ with }\varepsilon_p=
  \begin{cases}
    1&\text{if }p\equiv 1\pmod{4}\\
    i&\text{if }p\equiv 3\pmod{4}
  \end{cases}\]
by the evaluation of quadratic Gauss sums, so $\sqrt{p}\in\Z[\zeta_p,\zeta_4]\le\Z[\zeta_{4p}]$ and $\sqrt{p}\in\Z[\zeta_{4p}]_\lf^\times$ since $\ell\neq p$. The proposition is then a consequence of the construction and investigation of the $\ell$-adic Fourier transform by Deligne, Laumon and Brylinski: see \cite[Chapters 4, 5, 8]{KatzGKM}.
\end{proof}
\subsubsection{Point counting on families of curves}
  
\begin{proposition}\label{prop:hyperEllipticFamily}
  For $f\in\F_q[X]$ a squarefree polynomial of degree $2g\ge 2$, such that its set of zeros $Z_f$ is contained in $\F_q$, we consider the family of smooth projective hyperelliptic curves of genus $g$ parametrized by $z\in \F_q\backslash Z_f$ with affine models
  \[X_z : y^2=f(x)(x-z).\]

  Let $\lambda$ be an $\ell$-adic valuation on $\Oc=\Z[\zeta_{4p}]$ corresponding to a prime ideal $\lf$. There exists a geometrically irreducible sheaf of $\Oc_\lambda$-modules $\Fc=\Fc_{\lambda}$ on $\P^1_{\F_q}$ of generic rank $2g$, corresponding to a representation $\rho: \pi_{1,q}\to\GL(V)=\GL_{2g}(\Oc_\lambda)$ such that for all $z\not\in Z_f$,
    \[\frac{\det(1-q^{1/2}T\rho(\Frob_z))}{(1-T)(1-qT)}=Z(X_z,T):=\exp \left(\sum_{n\ge 1} |X_z(\F_{q^n})|\frac{T^n}{n}\right),\]
    \[t_\Fc(z)=\frac{q+1-|X_z(\F_{q})|}{q^{1/2}}\in \Oc_{q^{1/2}}\le\Oc_\lf.\]
  Moreover:
  \begin{enumerate}
  \item $\Sing(\Fc)=\{\infty\}\cup Z_f$ and $\Fc$ is everywhere tame. In particular, $\cond(\Fc)=2g+|Z_f|$.
  \item \label{item:hyperEllipticFamilyTransvection} At any $z\in Z_f$, the quotient $V/V^{I_z}$ is the trivial (one-dimensional) $I_z$--representation.
  \end{enumerate}
  
By reduction modulo $\lf$, this gives a sheaf of $\F_\lf$-modules on $\P^1_{\F_q}$ with the same properties and trace function $t_\Fc\pmod{\lf}$, where $\F_\lf=\Oc_\lambda/\lf\Oc_\lambda$.
\end{proposition}
\begin{proof}
  By \cite[Section 10.1]{KatzSarnak91} or \cite[Section 4]{Hall08} (using middle-convolutions), there exists a sheaf of $\Z_\ell$-modules on $\P^1_{\F_q}$ of generic rank $2g$, pointwise pure of weight $1$, such that for all $z\not\in Z_f$,
  \[Z(X_z,T)=\frac{\det(1-T\rho(\Frob_z))}{(1-T)(1-qT)}\text{ and }t_{\Fc}(z)=q+1-|X_z(\F_{q})| \ (z\not\in Z_f),\]
  along with properties (1) and (2) above. Normalizing by a Tate twist gives the sheaf with the desired properties.

\end{proof}

Other examples of families of curves are given in \cite[Chapter 10]{KatzSarnak91}.

\section{Probabilistic model}\label{sec:model}

Let $\Fc$ be a sheaf of $\F_\lf$-modules on $\P^1_{\F_q}$, lisse on an open $U$, and corresponding to a representation $\rho_\Fc: \pi_{1,q}\to\GL_n(\F_\lf)=\GL(V)$.

In this section, we setup a probabilistic model for short sums of the trace function $t_\Fc$ and show that it is accurate (with respect to density functions).

\subsection{Monodromy groups}

\begin{definition}
  The \textit{arithmetic and geometric monodromy groups} of $\Fc$ are the groups $G_\geom(\Fc)=\rho_\Fc\big(\pi_{1,q}^\geom\big)\le G_\arith(\Fc)=\rho_\Fc(\pi_{1,q})\le\GL_n(\F_\lf)$.
\end{definition}

\begin{definition}
  For $G=G_\arith(\Fc)$ or $G=G_\geom(\Fc)$, the inclusion $G\hookrightarrow\GL_n(\F_\lf)$ is called the \textit{standard representation} of $G$.
\end{definition}

  The determination of integral or finite monodromy groups is usually more difficult than that of monodromy groups over $\overline\Q_\ell$ (as Zariski closures of the images of the representations), because we consider simply subgroups of $\GL_n(\F_\lf)$ instead of (reductive\footnote{By a result of Deligne, the connected component at the identity of the geometric monodromy group of a pointwise pure of weight $0$ sheaf of $\overline\Q_\ell$-modules on $\P^1_{\F_q}$ is semisimple, see e.g. \cite[9.0.12]{KatzSarnak91}.}) algebraic subgroups of $\GL_n(\C)$.

\subsubsection{Examples}

Nonetheless, for the examples we consider:

  \begin{proposition}[Kummer sheaves]\label{prop:monoKummer}
    The arithmetic and geometric monodromy groups of a Kummer sheaf of $\F_\lf$-modules as in Proposition \ref{prop:Kummer} are equal to $\mu_d(\F_\lf)$, the group of $d$th roots of unity in $\F_\lf$.
  \end{proposition}
  \begin{proof}
    This is clear by the explicit construction of the Kummer sheaf (see e.g. \cite[Section 4.7]{DelEC}).
  \end{proof}
  
  The following two propositions extend the results over $\overline\Q_\ell$ from \cite{KatzGKM}, respectively \cite{KatzSarnak91}, and show that the finite monodromy groups are still as large as possible.
  
  \begin{proposition}[Kloosterman sheaves]\label{prop:monoKS}
    Let $n\ge 2$ be an integer and let $\Klc_n$ be the sheaf of $\F_\lf$-modules from Proposition \ref{prop:Klcn}. If $\ell\gg_n 1$ with $\ell\equiv 1\pmod{4}$, then
  \[G_\geom(\Klc_n)=G_\arith(\Klc_n)=
  \begin{cases}
    \SL_n(\F_\lf)&\text{if } n\text{ odd}\\
    \Sp_n(\F_\lf)&\text{if } n\text{ even.}
  \end{cases}
\]
If $p\equiv 1\pmod{4}$ and $\Oc$ is replaced by $\Z[\zeta_p]$, then this holds without restriction on $\ell\pmod{4}$.
\end{proposition}
\begin{proof}
  See \cite[Theorem 1.6]{PGIntMonKS16}.
\end{proof}

\begin{proposition}[Families of hyperelliptic curves]\label{prop:hyperEllipticFamilyMono}~
  In the setting of Proposition \ref{prop:hyperEllipticFamily}, assume that $\lf$ is completely split, i.e. $\F_\lf=\F_\ell$. For $\Fc$ the normalized sheaf of $\F_\ell$-modules on $\P^1_{\F_q}$ from Proposition \ref{prop:hyperEllipticFamily}, we have $G_\geom(\Fc)=G_\arith(\Fc)=\Sp_{2g}(\F_\ell)$.
\end{proposition}
\begin{proof}
  It follows from a theorem of Hall \cite[Section 5]{Hall08} and Proposition \ref{prop:hyperEllipticFamily} \ref{item:hyperEllipticFamilyTransvection} (i.e. the geometric monodromy group contains a transvection) that $G_\geom(\Fc)=\Sp_{2g}(\F_\ell)$.

  Since we normalized, \cite[Lemma 10.1.9]{KatzSarnak91} shows that the arithmetic monodromy group preserves the same pairing (without normalization, it is a group of symplectic similitudes with multiplicator $q$), so that $\Sp_{2g}(\F_\ell)=G_\geom(\Fc)\le G_\arith(\Fc)\le\Sp_{2g}(\F_\ell)$.
\end{proof}

\subsection{Model}\label{subsec:model}

We are interested in the $\F_\lf$-valued random variable
\begin{equation}
  \label{eq:RVt}
  \Big(t_\Fc(x)\Big)_{x\in\F_q}
\end{equation}
with respect to the uniform measure on $\F_q$.

Motivated by Chebotarev's density theorem/Deligne's equidistribution theorem (see e.g. \cite[Chapter 9]{KatzSarnak91}), the idea is to model the $G_\arith(\Fc)^\sharp$-valued random variable
\[\Big(\rho_\Fc^\sharp(\Frob_x)\Big)_{x\in U(\F_q)}\]
by the random variable $Y=\pi(X)$, where $X$ is uniformly distributed in $G_\arith(\Fc)$, $\pi: G_\arith(\Fc)\to G_\arith(\Fc)^\sharp$ is the projection to the conjugacy classes, and $\rho_\Fc^\sharp: \pi_{1,q}^\sharp\to G_\arith(\Fc)^\sharp$ is the natural map induced by $\rho_\Fc$.\\

We shall then naturally model \eqref{eq:RVt} by the random variable $Z=\tr(Y)$.

\subsubsection{Shifts}

\begin{definition}
  For $I\subset\F_q$, we define
  \[U_{\Fc,I}(\F_q)=\bigcap_{a\in I} (U_\Fc(\F_q)-a)=\F_q\backslash\bigcup_{a\in I}((\Sing(\Fc)\cap\F_q)-a),\]
  where $E-a=\{x-a : x\in E\}$ for any $E\subset\F_q$ and $a\in\F_q$.
\end{definition}

For $I\subset\F_q$ of size $L\ge 1$, we will then model the random vector
\begin{equation}
  \label{eq:modelShifts}
  \left(\left(\rho_\Fc^\sharp(\Frob_{x+a})\right)_{a\in I}\right)_{x\in U_{\Fc,I}(\F_q)}
\end{equation}
(with respect to the uniform measure on $U_{\Fc,I}(\F_q)$) by the random vector $(Y_1,\dots,Y_L)$, for $Y_i$ independent distributed like $Y$.

Correspondingly, we will model the random vector
$\left(\left(t_\Fc(x+a)\right)_{a\in I}\right)_{x\in\F_q}$ by $(Z_1,\dots,Z_L)$, for $Z_i$ independent distributed like $Z$.\\

Therefore, the sum of shifts
\[\Big(S(t_\Fc,I+x)\Big)_{x\in\F_q}=\left(\sum_{y\in I} t_\Fc(y+x)\right)_{x\in\F_q}\]
will be modeled by the random walk $S(L)=Z_1+\dots+Z_L$, as in \cite{LamzShortSums} and \cite{LamzModm} for multiplicative characters.

This is also to be compared with the model used in \cite{PG16} and \cite{PGGaussDistr16} for sheaves of $\overline\Q_\ell$-modules.

\begin{remark}
  When $x\in\Sing(\Fc)$, $\rho_\Fc^\sharp(\Frob_{x,q})$ is not a well-defined conjugacy class in $G_\arith(\Fc)$ (rather one in $\GL(V^{I_x})$). On the other hand, $t_\Fc(x)$ is a well-defined element of $\F_\lf$ for all $x\in\P^1(\F_q)$.
\end{remark}

\subsection{Coherent families}

We define a family of sheaves for which this model is accurate.

\subsubsection{Definition}

\begin{definition}
  Let $\Fc$ be a sheaf of $\F_\lf$-modules on $\P^1_{\F_q}$ corresponding to a representation $\rho_\Fc: \pi_{1,q}\to\GL_n(\F_\lf)$. If $\sigma\in\Aut(\F_\lf)$, we let $\sigma(\Fc)$ be the sheaf corresponding to the representation $\sigma\circ\rho_\Fc: \pi_{1,q}\to\GL_n(\F_\lf)\to\GL_n(\F_\lf)$.
\end{definition}

\begin{definition}[Coherent family]\label{def:coherent}
  Let $E$ be a number field and $\Lambda$ be a set of valuations on $E$. A family $(\Fc_\lambda)_{\lambda\in\Lambda}$, where $\Fc_\lambda$ is an irreducible sheaf of $\F_\lf$-modules over a finite field $\F_q=\F_{q(\lambda)}$, for $\lf$ the prime ideal corresponding to $\lambda$, is \textit{coherent} if:
  \begin{enumerate}
  \item[\ifamsart(1)\else1.\fi]\label{item:coherent1}(Conductor) $\cond(\Fc_\lambda)$ is uniformly bounded for $\lambda\in\Lambda$.
  \end{enumerate}
  and either:
  \begin{enumerate}[start=2]
  \item[\ifamsart(2)\else2.\fi]\textit{Kummer case}: There exists an integer $d\ge 2$ such that every $\Fc_\lambda$ is a Kummer sheaf with monodromy group $\mu_d(\F_\lf)$.
  \item[\ifamsart($2'$)\else$2'$.\fi] \textit{Cyclic simple case}: There exists a prime $d\ge 2$ such that for every $\lambda\in\Lambda$:
    \begin{enumerate}
    \item (Monodromy groups) The arithmetic and geometric monodromy groups of $\Fc_\lambda$ coincide and are equal to $\mu_d(\F_\lf)$.
    \item (Independence of shifts) There is no geometric isomorphism of the form $[+a]^*\Fc_\lambda\cong\Fc_\lambda^{\otimes i}$ for $1\le i<d$, $a\in\F_{q( \lambda)}^\times$.
    \end{enumerate}
  \item[\ifamsart($2''$)\else$2''$.\fi]\textit{Classical case}: There exists $G\in\{\SL_{n} : n\ge 2\}\bigcup\{\Sp_{n} : n\ge 2\text{ even}\}$ such that for every $\lambda\in\Lambda$:
    \begin{enumerate}
    \item (Monodromy groups) The geometric and arithmetic monodromy groups of $\Fc_\lambda$ coincide and are conjugate to $G(\F_\lf)$ in $\GL_n(\F_\lf)$ (with respect to the standard representation).
    \item (Independence of shifts) There is no geometric isomorphism of the form
      \begin{equation}
      \label{eq:translatinIsomFinite}
      [+a]^*\Fc_\lambda\cong \Lc\otimes\sigma(\Fc_\lambda)\text{ or }[+a]^*\Fc_\lambda\cong\Lc\otimes D(\sigma(\Fc_\lambda))
    \end{equation}
    for $a\in\F_{q(\lambda)}^\times$, $\sigma\in\Aut(\F_\lf)$ and $\Lc$ a sheaf of $\F_\lf$-modules of generic rank $1$ on $\P^1_{\F_{q(\lambda)}}$.
    \end{enumerate}
  \end{enumerate}
  We call $\mu_d$, resp. $\SL_n$ or $\Sp_n$ the \textit{monodromy group structure} of the family, and a \textit{bound on the conductor of the family} is any uniform bound for $\cond(\Fc_\lambda)$ ($\lambda\in\Lambda$).
\end{definition}

\begin{remark}\label{rem:modelableFixG}
  We fix the structure of the monodromy group, while we let $\F_q$ and $\F_\lf$ vary to study for example the reductions of a trace function whose values do not depend on $\lambda$, modulo various ideals (see Remark \ref{rem:tImageO}) as $q\to+\infty$. We could also let $d$, resp. $n$ vary (in the Kummer, resp. classical case), but this is not a natural aspect in applications. Nonetheless, we note that the implied constants will \textit{not} depend on $d$ in our results.
\end{remark}

\begin{remark}
  This is to be compared with the coherent families of sheaves of $\overline\Q_\ell$-modules on $\P^1_{\F_q}$ defined in \cite{PGGaussDistr16}.
\end{remark}

\subsubsection{Examples}\label{subsubsec:examplesCoherent}
\begin{proposition}[Kummer sheaves]\label{prop:KummerCoherent}
  Let $d\ge 2$ be an integer and let $\Lambda$ be a set of valuations of $\Q(\zeta_d)$. A family $(\Lc_{\chi(f_\lambda),\lambda})_{\lambda\in\Lambda}$ of Kummer sheaves, with monodromy group structure $\mu_d$ and $\deg(f_\lambda)$ bounded uniformly, is coherent.
\end{proposition}
\begin{proof}
  By Propositions \ref{prop:Kummer} and \ref{prop:monoKummer}, Conditions (1) and (2) of Definition \ref{def:coherent} are satisfied.
\end{proof}

\begin{proposition}[Kloosterman sheaves]\label{prop:KlCoherent}
  Let $n\ge 2$ be an integer and let $\Lambda$ be a set of valuations of $\Q(\zeta_{4p})$. We assume that every $\lambda\in\Lambda$ lies above a prime $\ell\gg_n 1$ with $\ell\equiv 1\pmod{4}$, as in Proposition \ref{prop:monoKS}. A family $(\Klc_{n,\lambda})_{\lambda\in\Lambda}$ of Kloosterman sheaves of rank $n$, as in Proposition \ref{prop:Klcn}, is coherent.
\end{proposition}
\begin{proof}
  By Propositions \ref{prop:Klcn} and \ref{prop:monoKS}, Conditions (1) and ($2''$a) of Definition \ref{def:coherent} are satisfied. It remains to show the independence of shifts (Condition ($2''$b)), which can be done exactly as in the $\overline\Q_\ell$ case in \cite[Section 7]{PGGaussDistr16} by analyzing the local ramification on both sides of an isomorphism of the form \eqref{eq:translatinIsomFinite}.
\end{proof}

\begin{proposition}[Families of hyperelliptic curves]\label{prop:hyperellipticCoherent}
  Let $f\in\Z[X]$ be a squarefree polynomial of degree $2g\ge 2$ and let $\Lambda$ be a set of valuations of $\Q(\zeta_{4p})$ of degree $1$ (i.e. corresponding to completely split ideals). A family $(\Fc_\lambda)_{\lambda\in\Lambda}$ of sheaves of $\F_\ell$-modules with respect to the reductions of $f$ as in Proposition \ref{prop:hyperEllipticFamily} is coherent.
\end{proposition}
\begin{proof}
  By Propositions \ref{prop:hyperEllipticFamily} and \ref{prop:hyperEllipticFamilyMono}, Conditions (1) and ($2''$a) of Definition \ref{def:coherent} are satisfied, and it suffices to verify the independence of shifts (Condition ($2''$b)). Again, the argument is the same as the one over $\overline\Q_\ell$ in \cite[Section 7]{PGGaussDistr16}.
\end{proof}

\subsection{Accuracy of the model}

\begin{table}[h!]
  \centering
  \begin{tabular}{c|cccccc}
    $G$&$\dim G$&$\rank G$&$\alpha(G)$&$\beta_+(G)$&$\beta_-(G)$\\\hline
    $\SL_n$&$n^2-1$&$n-1$&$\frac{n^2-1}{2}$&$\frac{n^2+n-2}{2}$&$\frac{n(n-1)}{2}$\\
    $\Sp_{n}$ ($n$ even)&$\frac{n(n+1)}{2}$&$\frac{n}{2}$&$\frac{n(n+2)}{8}$&$\frac{n(n+2)}{4}$&$\frac{n^2}{4}$
  \end{tabular}
  \caption{Constants for the groups considered.}
  \label{table:dimrank}
\end{table}

\begin{definition}
  Let $(\Fc_\lambda)_{\lambda\in\Lambda}$ be a coherent family with monodromy group structure $G$ and let $I\subset\F_q$. We say that a sheaf of $\F_\lf$-modules $\Fc_\lambda$ on $\P^1_{\F_q}$ in the family is \emph{$I$-compatible} if either:
  \begin{itemize} 
  \item The family is in the cyclic simple case or the classical case (cases ($2'$) and ($2''$) of Definition \ref{def:coherent}).
  \item The family is in the Kummer case (case (2) of Definition \ref{def:coherent}), so that $\Fc_\lambda=\Lc_{\chi(f)}$ for $\chi:\F_q^\times\to\C^\times$ a character of order $d$, $f=f_1/f_2\in\F_q(X)$ with $(f_1,f_2)=1$, and we have that
    \begin{equation}
      \label{eq:coherentCondition}
      \sum_{i=1}^m x_i\neq 0\text{ for all }x_1,\dots,x_m\in I, 1\le m\le\deg(f_1).
    \end{equation}
  \end{itemize}  
\end{definition}

\begin{example}\label{ex:noZeromSum}
  Condition \eqref{eq:coherentCondition} holds if $\deg(f)=1$ or if for an arbitrary $\F_p$-basis of $\F_q$ with coordinates $\pi_i: \F_q\to \{1,\dots,p\}$ we have $\max_{a\in I} \pi_i(a)<p/\deg(f_1)$ for some $1\le i\le e$.
\end{example}

\begin{definition}
  Let $L\ge 1$ be an integer and $G$ be the monodromy group structure of a coherent family. We define
  \[E(G,L,\F_\lf)=
  \begin{cases}
    |\F_\lf|^{L\beta_+(G)+2\beta_-(G)}&G\text{ classical}\\
    d^L&G=\mu_d, \ d\text{ prime}\\
    d^{L+1}&\text{otherwise}
  \end{cases}
\]
  with $\beta_\pm(G)=(\dim{G}\pm\rank{G})/2$, given in Table \ref{table:dimrank}.
\end{definition}

\begin{theorem}\label{thm:model}
  Let $(\Fc_\lambda)_{\lambda\in\Lambda}$ be a coherent family with monodromy group structure $G$. For $\lambda\in\Lambda$, let $I\subset\F_q=\F_{q(\lambda)}$ be of cardinality $L$ and $h: (G(\F_\lf)^\sharp)^I\to\R$ be any function. If $\Fc=\Fc_\lambda$ is an $I$-compatible sheaf on $\P^1_{\F_q}$, then
  \[\E \left[h\Big((\rho_\Fc^\sharp(\Frob_{x+a}))_{a\in I}\Big)\right]=\E\left(h(Y_1,\dots,Y_L)\right)+O\Big(L ||h||_\infty q^{-1/2}E(G,L,\F_\lf)\Big),\]
where the random variables $Y_i$ and probability spaces are as in Section \ref{subsec:model}, $||h||_\infty=\max_{x\in (G(\F_\lf)^\sharp)^I} |h(x)|$. Moreover, if $h$ takes values in $\R_{\ge 0}$, then the above is
\[\E\left(h(Y_1,\dots,Y_L)\right)\left(1+O(L q^{-1/2}E(G,L,\F_\lf))\right).\]
The implied constants depend only on the monodromy group structure and a bound on the conductor of the family.
\end{theorem}

\begin{remark}
  When $I=\{0\}$, this is Chebotarev's theorem as it appears for example in \cite{KowRank06}.
\end{remark}

\begin{corollary}\label{cor:model}
  Under the hypotheses of Theorem \ref{thm:model}, for any function $h: \F_\lf^I\to\R$, we have
  \[\E \left[h\Big((t_\Fc(x+a))_{a\in I}\Big)\right]=\E\left(h(Z_1,\dots,Z_L)\right)+O\Big(L ||h||_\infty q^{-1/2}E(G,L,\F_\lf))\Big),\]
  where the random variables $Z_i$ are probability spaces are as in Section \ref{subsec:model}. If $h$ takes values in $\R_{\ge 0}$, then the above is
  \[\E\left(h(Z_1,\dots,Z_L)\right)\left(1+O(L q^{-1/2}E(G,L,\F_\lf))\right).\]
  The implied constants depend only on the monodromy group structure and a bound on the conductor of the family.
\end{corollary}

\begin{remark}
  Note that we must take $L<q^{1/2}$ to have $E(G,L,\F_\lf)L=o(q^{1/2})$ as $q\to+\infty$.
\end{remark}

\subsection{Comments on the ranges}\label{subsec:limitationRange}

Let us consider the above in the context of Section \ref{subsec:reductions} and Remark \ref{rem:tImageO}, i.e. when the sheaf of $\F_\lf$-modules on $\P^1_{\F_q}$ from Theorem \ref{thm:model} arises from the reduction of a sheaf of $\Z[\zeta_d]_\lambda$-modules, allowing to study the reduction of a trace function $t: \F_q\to\Z[\zeta_d]$ modulo various ideals.

By Remark \ref{rem:splittingCycl}, recall that if $\F_\lf$ is the residue field of $\Z[\zeta_d]_\lambda$ at some prime ideal above $\ell$, then $d<|\F_\lf|=\ell^{m}$, for $m$ the order of $\ell$ in $(\Z/d)^\times$.
\subsubsection{Choice of the parameters}
Thus, we may want to choose our parameters $(q,\ell,\lf,d)$ so that $d<|\F_\lf|<|\F_q|=p^e$. Given $p$, $\ell$ and $d$, this is holds true for any $\lf$ above $\ell$ if $e\ge \frac{\varphi(d)\log\ell}{\log{p}}$.

\subsubsection{Limitation}

Together with the condition 
\[L\ll
  \begin{cases}
    \frac{\log{q}}{\log{|\F_\lf|}}&\text{if }G\text{ classical}\\
    \frac{\log{q}}{\log{d}}&\text{if }G=\mu_d
  \end{cases}
\]
from Theorem \ref{thm:model}, the relation $d<|\F_\lf|$ implies that $L\ll e$ if $G$ is classical and $d=p$ (e.g. for Kloosterman sums). Hence, we must in this case take $e$ large enough with respect to $L$, which is a limitation of the method to keep in mind. Note however that it is not unusual to encounter results stated in fixed characteristic with the degree $e$ going to infinity (see e.g. \cite[Chapter 9]{KatzSarnak91} and \cite[Chapter 3]{KatzGKM}).

\section{Proof of Theorem \ref{thm:model} and Corollary \ref{cor:model}}\label{sec:proofThmModel}

In the following, we use the notations of Theorem \ref{thm:model} and we let $I=\{a_1,\dots,a_L\}\subset\F_q$, $V(\F_q)=U_{\Fc,I}(\F_q)$.

\begin{definition}
  For $\bs v\in(G^\sharp)^I=(v_1,\dots,v_L)$ and $\widehat G$ the set of characters of irreducible complex representations of $G$, we define
  \[E(\bs v)=\sum_{\substack{\chi_1,\dots,\chi_L\in \widehat G\\\textnormal{not all trivial}}}\left(\prod_{i=1}^L\overline\chi_i(v_i)\right)\frac{1}{|V(\F_q)|}\sum_{x\in V(\F_q)} \prod_{i=1}^L\chi_i(\rho_\Fc^\sharp(\Frob_{x+a_i})).\]
\end{definition}

We start with the following relation between the expected values we are interested in:
\begin{proposition}\label{prop:EModel}
  For $I=\{a_1,\dots,a_L\}\subset\F_q$, the expected value
  \begin{equation}
    \label{eq:EModel}
    \E \left[h\Big((\rho_\Fc^\sharp(\Frob_{x+a}))_{a\in I}\Big)\right]
  \end{equation}
  is given by
  \[\E(h(Y_1,\dots,Y_L))+O \left(||h||_\infty\max_{\bs v\in (G^\sharp)^I} |E(\bs v)|\right).\]
  When $h$ is nonnegative, \eqref{eq:EModel} is also
  \[\E(h(Y_1,\dots,Y_L))\left[1+O\left(\max_{\bs v\in (G^\sharp)^I} |E(\bs v)|\right)\right].\]
\end{proposition}
\begin{proof}
  By definition, the left-hand side of \eqref{eq:EModel} is
  \[\frac{1}{|V(\F_q)|}\sum_{x\in V(\F_q)} h(\rho_\Fc^\sharp(\Frob_{x+a_1}),\dots,\rho_\Fc^\sharp(\Frob_{x+a_L}))=\sum_{\bs v\in (G^\sharp)^I} h(\bs v)\frac{|D(\bs v)|}{|V(\F_q)|},\]
  where
  \[D(\bs v)=\{x\in V(\F_q) : \rho_\Fc^\sharp(\Frob_{x+a_i})=v_i \ (1\le i\le L)\}  \]
  for $\bs v=(v_1,\dots,v_L)$. By Schur's orthogonality relations for the finite group $G$,
  \begin{eqnarray*}
    \frac{|D(\bs v)|}{|V(\F_q)|}&=&\frac{1}{|V(\F_q)|}\sum_{x\in V(\F_q)} \prod_{i=1}^L \delta_{\rho_\Fc^\sharp(\Frob_{x+a_i})=v_i}\\
    &=&\frac{1}{|V(\F_q)|}\sum_{x\in V(\F_q)} \prod_{i=1}^L \left(\frac{|v_i|}{|G|}\sum_{\chi\in \widehat G} \chi(\rho_\Fc^\sharp(\Frob_{x+a_i}))\overline\chi(v_i)\right)\\
                        &=&\frac{\prod_{i=1}^L |v_i|}{|G|^L}\sum_{\chi_1,\dots,\chi_L\in \widehat G}\frac{1}{|V(\F_q)|}\sum_{x\in V(\F_q)} \prod_{i=1}^L\chi_i(\rho_\Fc^\sharp(\Frob_{x+a_i}))\overline\chi_i(v_i)\\
                        &=&\frac{\prod_{i=1}^L |v_i|}{|G|^L}\left(1+E(\bs v)\right),
  \end{eqnarray*}
  where $|v|$ denotes the size of a conjugacy class $v\in G^\sharp$. On the other hand,
  \begin{equation*}
    \E(h(Y_1,\dots,Y_L))=\frac{1}{|G|^L}\sum_{\bs v\in (G^\sharp)^I}h(\bs v)\prod_{i=1}^L |v_i|.
\end{equation*}
\end{proof}

To prove Theorem \ref{thm:model}, we thus want to show that the expression
\begin{equation}
  \label{eq:sumProducts}
  \sum_{x\in V(\F_q)} \prod_{i=1}^L\chi_i(\rho_\Fc^\sharp(\Frob_{x+a_i}))
\end{equation}
in $E(\bs v)$ is small for $\chi_1,\dots,\chi_L\in\hat G$ not all trivial and $\bs v\in (G^\sharp)^I$.

\subsection{Reinterpretation of \eqref{eq:sumProducts}}

\begin{definition}
  We fix an isomorphism of fields $\iota: \overline\Q_\ell\to\C$. For $\eta: G\to\GL(V)$ a complex representation, we let $\Fc_\eta$ be the sheaf of $\overline\Q_\ell$-modules on $\P^1_{\F_q}$ corresponding to the representation
  \[\iota^{-1}\circ\eta\circ\rho_\Fc: \pi_{1,q}\to G\to \GL(V)\to\GL(\iota^{-1}(V)).\]
\end{definition}
\begin{remark}
      Since $G$ is discrete, there are no issues with the continuity of the composition $\iota^{-1}\circ\eta\circ\rho_\Fc$, even if $\iota$ is not continuous.
\end{remark}

Note that the trace function of $\Fc_\eta$ at unramified points is
\[\chi\circ\rho_\Fc^\sharp\circ\Frob: U_{\Fc_{\eta}}(\F_q)\to \pi_{1,q}^\sharp\to G(\F_\lf)^\sharp\to\C,\]
where $\chi$ is the character of $\eta$. Thus, we can rewrite \eqref{eq:sumProducts} as
\begin{equation}
  \sum_{x\in V(\F_q)} \prod_{i=1}^L t_{\Fc_i}(x)\text{, where }\Fc_i=[+a_i]^*\Fc_{\eta_i} 
\end{equation}
is a sheaf of $\overline\Q_\ell$-modules on $\P^1_{\F_q}$, for $\eta_i$ the representation corresponding to $\chi_i$.

\subsection{Sums of products of trace functions}\label{subsec:sumProducts}

The estimation of sums of products of the form $\sum_{x\in\F_q} \prod_{i=1}^H t_i(x)$, for $t_i$ the trace function of a sheaf of $\overline\Q_\ell$-modules on $\P^1_{\F_q}$, is precisely the question surveyed in \cite{FKMSumProducts}.

We need the following estimate, where the dependency with respect to the conductors is precisely tracked:
\begin{proposition}\label{prop:sumProducts}
  Let $(\Fc_i)_{1\le i\le L}$ be a tuple of pointwise pure of weight $0$ sheaves of $\overline\Q_\ell$-modules on $\P^1_{\F_q}$, with corresponding trace functions $(t_i: \F_q\to\C)_{1\le i\le L}$. We assume that the arithmetic and geometric monodromy groups of $\Gc=\bigoplus_{i=1}^L \Fc_i$ coincide and are as large as possible, i.e. isomorphic to $\prod_{i=1}^L G_\geom(\Fc_i)$. Then, for $S=\bigcup_{i=1}^L \Sing(\Fc_i)$,
\[\sum_{x\in\F_q\backslash S} t_1(x)\dots t_L(x)=q\prod_{i=1}^L \dim(\Fc_i)_{\pi_{1,q}^\geom}+O(E\sqrt{q})\]
with an absolute implied constant, where
\[E\ll\sqrt{q}\left(\prod_{i=1}^L\rank(\Fc_i)\right)\left(|S|+\sum_{x\in S}\sum_{i=1}^L \Swan_x(\Fc_i)\right).\]
\end{proposition}
\begin{proof}
  Let $\Fc=\bigotimes_{i=1}^L \Fc_i$, which satisfies $\Sing(\Fc)\subset S$. For any embedding $\iota: \overline\Q\to\C$,
  \begin{eqnarray*}
    \sum_{x\in\F_q\backslash S} t_1(x)\dots t_L(x)&=&\sum_{x\in U_\Fc(\F_q)} t_\Fc(x) + O \left(||t_1||_{\iota,\infty}\dots ||t_L||_{\iota,\infty}|S|\right)\\
                                                  &=&\sum_{x\in U_\Fc(\F_q)} t_\Fc(x)+O\left(\rank(\Fc)|S|\right).
  \end{eqnarray*}
  By \cite[Lemma 1.3]{KatzGKM}, we also have $\Swan_x(\Fc)\le \rank(\Fc)\sum_{i=1}^L \Swan_{x}(\Fc_i)$ for any $x\in\Sing(\Fc)$. Thus, Theorem \ref{thm:sumTF} yields
    \begin{eqnarray*}
    \sum_{x\in\F_q\backslash S} t_1(x)\dots t_L(x)&=&q\cdot\tr\left(\Frob_q\,\mid\,\Fc_{\pi_{1,q}^\geom}\right)\\
                                      &&+O \left(\sqrt{q}\rank(\Fc)\left(|S|+\sum_{x\in S}\sum_{i=1}^L\Swan_{x}(\Fc_i)\right)\right).
    \end{eqnarray*}
  If $\Fc_i$ corresponds to the representation $\rho_{i}: \pi_{1,q}\to\GL_{n_i}(\overline\Q_\ell)$ with arithmetic monodromy group $G_i$, we let $\Std_i: G_i\to\GL_{n_i}(\overline\Q_\ell)$ be the standard representation of $G_i$. If $\rho_\Gc$ is the representation of $\pi_{1,q}$ corresponding to $\Gc$, then the sheaf $\Fc$ corresponds to the representation $\Lambda\circ\rho_\Gc$ for $\Lambda=\bigboxtimes_{i=1}^L \Std_i$ and
  \[\Fc_{\pi_{1,q}}\cong\Lambda_{G_\arith(\Gc)}, \ \Fc_{\pi_{1,q}^\geom}\cong\Lambda_{G_\geom(\Gc)},\]
  where $\Lambda_{H}$ is the space of coinvariants of $\Lambda$ with respect to the action of a subgroup $H\le G_\arith(\Gc)$. If the geometric monodromy group of $\Gc$ is as large as possible, then
\[\Fc_{\pi_{1,q}^\geom}\cong\Lambda_{G_\geom(\Gc)}\cong\bigotimes_{i=1}^L (\Std_i)_{G_\geom(\Fc_i)}.\]
If moreover $G_\arith(\Gc)=G_\geom(\Gc)$, the Frobenius acts trivially on $\Fc_{\pi_{1,q}^\geom}\cong\Lambda_{G_\geom(\Gc)}=\Lambda_{G_\arith(\Gc)}$ and
\begin{eqnarray*}
  \tr\left(\Frob_q\,\mid\,\Fc_{\pi_{1,q}^\geom}\right)&=&\dim(\Fc_{\pi_{1,q}^\geom})=\prod_{i=1}^L \dim(\Std_i)_{G_\geom(\Fc_i)}.
\end{eqnarray*}
\end{proof}
This leads to the following estimates for \eqref{eq:sumProducts}:
\begin{proposition}\label{prop:sumProductsTranslatesFinite}
  For $L\ge 1$, let $a_1,\dots,a_L\in\F_q$ be distinct, and let $\eta_i$ be complex irreducible representations of $G$, not all trivial, with characters $\chi_i$ $(1\le i\le L)$. We assume that one of the following holds:
  \begin{enumerate}
  \item\label{item:sumProductsTranslatesFinite1} The arithmetic and geometric monodromy groups of $\bigoplus_{1\le i\le L}[+a_i]^*\Fc_{\eta_i}$ coincide and are as big as possible, i.e. isomorphic to $\prod_{1\le i\le L} G/\ker\eta_i$, or
\item $\Fc$ is a $\{a_1,\dots,a_L\}$-compatible Kummer sheaf $\Lc_{\chi(f)}$.
\end{enumerate}
Then, if $L\ll q$,
\[\max_{\bs v\in (G^\sharp)^I} |E(\bs v)|\ll \frac{\cond(\Fc)^2|G|^\delta}{\sqrt{q}}\sum_{\chi_1,\dots,\chi_L\in G^\sharp}\left(\prod_{i=1}^L \dim\eta_i\right)^2\sum_{i=1}^L \dim\eta_i,\]
with $\delta=0$ in case \ref{item:sumProductsTranslatesFinite1} and $\delta=1$ otherwise.
\end{proposition}
\begin{proof}
  It suffices to show that the sum of products \eqref{eq:sumProducts} is
  \[\ll \sqrt{q}\cond(\Fc)^2|G|^\delta\left(\prod_{i=1}^L \dim\eta_i\right)\sum_{i=1}^L \dim\eta_i\]
  whenever $\chi_1,\dots,\chi_L$ are not all trivial, since we then have
  \[|E(\bs v)|\le\cond(\Fc)^2|G|^\delta\frac{\sqrt{q}}{|V(\F_q)|} \sum_{\substack{\chi_1,\dots,\chi_L\in \widehat G\\\textnormal{not all trivial}}}\left(\prod_{i=1}^L\dim\eta_i\right)^2\sum_{i=1}^L \dim\eta_i\]
  and $\sqrt{q}/|V(\F_q)|\le q^{-1/2}|1-q^{-1}L\cond(\Fc)|$.
  \begin{enumerate}[leftmargin=*]
  \item We note that $\Fc_\eta$ is geometrically irreducible, $\rank(\Fc_\eta)=\dim\eta$, $\Sing(\Fc_\eta)$ is contained in $\Sing(\Fc)$ and $\Swan_x(\Fc_\eta)\le\dim\eta\Swan_x(\Fc)$ for all $x\in\P^1(\overline\F_q)$ (by \cite[3.6.2]{KatzGKM}). By Proposition \ref{prop:sumProducts} applied with $\Fc_i=[+a_i]^*\Fc_{\eta_i}$, the sum \eqref{eq:sumProducts} is thus
    \[\prod_{1\le i\le L}\dim(\Fc_{\eta_i}^{G})+O\left(q^{-1/2}\cond(\Fc)^2\left(\prod_{i=1}^L \dim\eta_i\right)\sum_{i=1}^L \dim\eta_i\right).\]
    By Schur's Lemma, $\dim(\Fc_{\eta_i}^{G})$ is equal to $1$ if $\eta_i$ is trivial, zero otherwise.
  \item For every $i=1,\dots,L$, there exists an integer $0\le b_i<d$ such that $\eta_i$ is the one-dimensional representation $x\mapsto x^{b_i}$. By multiplicativity,
  \[\frac{1}{q}\sum_{x\in\F_q} \prod_{i=1}^L\chi_i(\rho_\Fc^\sharp(\Frob_{x+a_i}))=\aver{x}{\F_q}{q} t_\Gc(x)\]
  where $\Gc=\Fc_{\chi(g)}$ with $g(X)=\prod_{i=1}^L f(X+a_i)^{b_i}$. Since $\cond(\Gc)\le 1+\deg(g)\le 1+Ld\deg(f)$, Corollary \ref{cor:sumTF} gives
  \[\aver{x}{\F_q}{q} t_\Gc(x)=\delta_{g\text{ is a }d-\text{power}} +O(Ld\deg(f)q^{-1/2}).\]
  As in \cite[Section 4]{PGGaussDistr16}, we see that, under the compatibility assumption, $g$ cannot be a $d$-power unless all $b_i$ are zero.
  \end{enumerate}
\end{proof}
\subsection{Finite Goursat-Kolchin-Ribet criteria}

It remains to determine when Hypothesis \ref{item:sumProductsTranslatesFinite1} of Proposition \ref{prop:sumProductsTranslatesFinite} holds. For sheaves of $A=\overline\Q_\ell$-modules, this is handled by the Goursat-Kolchin-Ribet criterion of Katz (see \cite{FKMSumProducts}). We give here an analogue for sheaves of $\F_\lf$-modules.

\subsubsection{Preliminaries}

First, recall the classical Goursat Lemma:
\begin{lemma}[Goursat]\label{lemma:Goursat}
  Let $G_1,G_2$ be groups (resp. Lie algebras) and $H\le G_1\times G_2$ be a subgroup (resp. Lie subalgebra) such that the two projections $p_i: H\to G_i$ $(i=1,2)$ are surjective. 
\begin{equation*}
  \xymatrix@R=0.4cm{
    &G_1\ar[r]&G_1/\ker p_2\\
    H\ar[r]\ar[ur]_{p_1}\ar[dr]^{p_2}&G_1\times G_2\ar[r]\ar[u]\ar[d]&(G_1/\ker p_2)\times (G_2/\ker p_1)\ar[u]\ar[d]\\
    &G_2\ar[r]&G_2/\ker p_1
  }
\end{equation*}
Then the image of $H$ in $G_1/\ker p_2\times G_2/\ker p_1$ is the graph of an isomorphism $G_1/\ker p_2\cong G_2/\ker p_1$. In particular, if $G_1,G_2$ are simple, then either $H=G_1\times G_2$, or $H$ is the graph of an isomorphism $G_1\cong G_2$.
\end{lemma}
\begin{proof}
  See for example \cite[Lemma 5.2.1]{Ribet76}.
\end{proof}

\begin{lemma}[{\cite[Lemma 5.2.2 and p. 791]{Ribet76}}]\label{lemma:Ribet} Let $G_1,\dots,G_n$ be nontrivial finite groups with no proper nontrivial abelian quotients and let $G\le G_1\times\dots\times G_n$ be such that every projection $G\to G_i\times G_j$ $(i\neq j)$ is surjective. Then $G=G_1\times\dots\times G_n$.
\end{lemma}
\begin{proof}
  In \cite{Ribet76}, this is stated when $G_i$ has no nontrivial abelian quotient, and the condition is used at the end of the proof of the sublemma, \cite[p. 764]{Ribet76}. In the notations of the latter, $B'/K'$ is abelian, and if $B'=S_n$ has no \textit{proper nontrivial} abelian quotient, then either $K'=B'$ and one can conclude, or $K'$ is trivial, which implies that $S_n$ is trivial.
\end{proof}

\begin{definition}
  Let $k$ be a field. A pair $(G_i\to\GL(V_i))_{i=1,2}$ (or $(G_i\to\PGL(V_i))_{i=1,2}$) of faithful group representations over $k$ is \textit{Goursat-adapted} if every isomorphism $G_1\cong G_2$ is of the form
\[\begin{cases}
    X\mapsto A\sigma(X)A^{-1}&\text{for an isomorphism }A: V_1\to V_2\text{ or}\\
    X\mapsto A\sigma(X)^{-t}A^{-1}&\text{for an isomorphism }A: V_1^*\to V_2
  \end{cases}\]
  with $\sigma\in\Aut(k)$, $\sigma=\id$ unless $k$ is finite.
\end{definition}

\begin{example}\label{ex:GoursatSLSp}
  Let $G\in\{\PSL_{n}(k) : n\ge 2\}\bigcup\{\PSp_{n/2}(k) : n\ge 2\text{ even}\}$ for $k$ be a finite field. If $G_1,G_2$ are conjugate to $G$, then $(G_i,\Std)_{i=1,2}$ is Goursat-adapted, where $\Std$ is the natural embedding in $\PGL_n(k)$. Indeed, by \cite[4.237]{Gor82} and \cite[Theorem 12.5.1]{Carter72}, if $G$ is a finite simple group of Lie type defined over $k$, every automorphism can be written as the product of an inner, graph, diagonal, and field automorphism; more precisely,
  \[\Out(G)\cong \left(\mathrm{Diag}(G)\Aut(k)\right).\Graph(G),\]
  where $\mathrm{Diag}(G)$ (resp. $\Graph(G)$) is the group of diagonal automorphisms (resp. the group of graph automorphisms of the corresponding Dynkin diagram). But for $n\ge 2$, $\Graph(A_{n-1})\cong\Z/2$ (corresponding to the transpose-inverse map) while $\Graph(C_n)$ is trivial, with the standard nomenclature for Dynkin diagrams.
\end{example}

\subsubsection{Finite groups of Lie type}

\begin{proposition}[Goursat-Kolchin-Ribet for quasisimple groups]\label{prop:finiteGKR}
  Let $\pi$ be a topological group, $k$ be a finite field, $F$ be a field, and for $i=1,\dots, N$, let $\rho_i: \pi\to\GL(V_i)$ be a finite-dimensional representation over $k$ with finite monodromy group $G_i=\rho_i(\pi)\le\GL(V_i)$, and $\eta_i: G_i\to\GL(W_i)$ be a nontrivial representation over $F$.

  We consider the representation $\rho=\oplus_{i=1}^N(\eta_i\circ\rho_i): \pi\to\GL(\oplus_{i=1}^N W_i)\cong \prod_{i=1}^N\GL(W_i)$ with monodromy group $G=\rho(\pi)$.

\begin{equation*}
  \xymatrix@R=0.4cm{
    \pi\ar[r]^{\rho_i}\ar[d]^\rho&G_i\ar[r]^{\hspace{-0.3cm}\eta_i}&\eta_i(G_i)\ar[r]&\GL(W_i)\\
    G\ar[d]&&&&\\
    \prod_{i=1}^N \eta_i(G_i)\cong\prod_{i=1}^N (G_i/\ker\eta_i)\ar[d]&&&&\\
    \prod_{i=1}^N \GL(W_i)\ar@/_2pc/[uuurrr]&&&&
  }
\end{equation*}
We assume that:
\begin{enumerate}
\item\label{item:finiteGKR1} The groups $G_i$ are quasisimple, i.e. they are perfect $(G_i=G_i^\der)$ and $G_i'=G_i/Z(G_i)$ is simple.
\item\label{item:finiteGKR2} For every $i\neq j$, $\big(G_l'\to\PGL(V_l)\big)_{l=i,j}$ is Goursat-adapted.
\item \label{item:finiteGKRNoIsom} For every $i\neq j$, there is no isomorphism
  \[\rho_i\cong\chi\otimes\sigma(\rho_j)\text{ or }\rho_i\cong \chi\otimes D(\sigma(\rho_j))\]
   for $\chi$ a 1-dimensional representation of $\pi$ over $k$ and $\sigma\in\Aut(k)$.
\end{enumerate}
Then $G$ is as large as possible, i.e. $G=\prod_i (G_i/\ker\eta_i)$.
\end{proposition}
\begin{proof}
    Since $G_i$ is quasisimple, note that we have either:
  \begin{itemize}
  \item $G_i=Z(G_i)\ker\eta_i$. By taking derived subgroups, this gives $G_i^\der=G_i\le(\ker\eta_i)^\der\le\ker\eta_i$, so $\ker\eta_i=G_i$ and $\eta_i$ is trivial, which is excluded;
  \item $\ker\eta_i\le Z(G_i)$.
  \end{itemize}
  
  For $H$ any group, let us continue to denote $H'=H/Z(H)$. By perfectness, it is enough to show that $G'=\prod_i (G_i/\ker\eta_i)'\cong\prod_i G_i'$.

  Since a quasisimple group has no nontrivial abelian quotient (the derived subgroup is the smallest normal subgroup with an abelian quotient), it is enough to treat the case $n=2$ by Lemma \ref{lemma:Ribet}.

  By Goursat's Lemma \ref{lemma:Goursat} and the simplicity of $G_i'$, either $G'=G_1'\times G_2'$, or $G'$ is the graph of an isomorphism $G_1'\cong G_2'$. In the second case, since the center of $\GL$ is the group of scalar matrices, the isomorphism given by hypothesis \ref{item:finiteGKR2} lifts to an isomorphism contradicting \ref{item:finiteGKRNoIsom}.
\end{proof}

\begin{remarks}\label{rem:settingfiniteGKR}
  Proposition \ref{prop:finiteGKR} should be compared with Katz's version over an algebraically closed field \cite[1.8.2]{KatzESDE}. Here, we more generally compute the monodromy group of $\oplus_{i=1}^N (\eta_i\circ\rho_i)$ instead of $\oplus_{i=1}^N \rho_i$, while still assuming Condition \ref{item:finiteGKRNoIsom} only on $\rho_i$ (and not on $\eta_i\circ\rho_i$). Over an algebraically closed field, the 1-dimensional representations appear when passing from $G$ to $G^{0,\der}$, while in Proposition \ref{prop:finiteGKR} they appear when passing from $G$ to $G'$. Moreover, the assumption of quasisimplicity here plays the role of the semisimplicity hypothesis in \cite{KatzESDE}.
\end{remarks}

\begin{example}
 Let $k$ be a finite field and $n\ge 1$ be an integer. By \cite[Theorem 24.17]{TesMal11}, $\SL_n(k)$ and $\Sp_{2n}(k)$ with their standard representations are quasisimple as soon as $|k|>3$. Hence, by Example \ref{ex:GoursatSLSp}, conditions \ref{item:finiteGKR1} and \ref{item:finiteGKR2} of Proposition \ref{prop:finiteGKR} hold if there exists $G\in\{\PSL_{n}(k): n\ge 2\}\cup\{\PSp_{2n}(k) : n\ge 1\}$ such that every $G_i$ is conjugate to $G$.
\end{example}

\subsubsection{Roots of unity}

Lastly, we give a version of the Goursat-Kolchin-Ribet criterion for cyclic groups of prime order.

\begin{proposition}\label{prop:finiteGKRmud}
For $i=1,\dots,L$, let $\rho_i: \pi\to k^\times$ be a one-dimensional representations over a field $k$ of a topological group $\pi$, with monodromy group $G_i=\rho_i(\pi)\cong\Z/d$ ($d$ prime), and let $\eta_i: G_i\to F^\times$ be a nontrivial representation over a field $F$. Consider the representation $\rho=\oplus_i(\eta_i\circ\rho_i): \pi\to\prod_i F^\times$ with monodromy group $G=\rho(\pi)$. If there is no isomorphism of the form $\rho_i\cong\rho_j^{\otimes a}$ for $i\neq j$, $1\le a<d$, then $G$ is as large as possible, i.e. $G\cong\prod_i \Z/d$.
\end{proposition}
\begin{proof}
  Since $\Z/d$ is simple, we can apply Lemma \ref{lemma:Ribet} to reduce to the case of two representations as before. By Goursat's Lemma \ref{lemma:Goursat}, either $G$ is as large as possible, or it is the graph of an isomorphism $G_1\to G_2$. Since $\Aut(\Z/d)\cong(\Z/d)^\times$, this proves the statement.
\end{proof}

\subsection{Sums of dimensions of irreducible representations}
\begin{definition}
  For a finite group $G$ and $m\ge 1$ an integer, we define $d_m(G)=\sum_{\chi\in\widehat G} (\dim\chi)^m$.
\end{definition}
\begin{lemma}\label{lemma:sizesGEell}
  For any finite group $G$, $d_1(G)\le|G|^{1/2}|G^\sharp|^{1/2}$, $d_2(G)=|G|$ and for every $m\ge 3$, $d_m(G)\le|G|^{m/2}|G^\sharp|$. Moreover:
  \begin{enumerate}
  \item If $G$ is abelian, $d_m(G)=|G|$ for every $m\ge 1$.
  \item If $G\le\GL_n(k)$ is a finite classical group of Lie type  over the finite field $k$, we have\footnote{Here, the notation $f_1(G)=\Theta_n(f_2(G))$ means that there exists constants $C_1(n),C_2(n)>0$ depending only on $n$ such that $C_1(n)f_2(G)\le f_1(G)\le C_2(n)f_2(G)$ for all $G$.}
    $d_1(G)\ll_n|k|^{\frac{\dim{G}+\rank{G}}{2}}$, $d_2(G)=\Theta_n(|k|^{\dim G})$, $|G^\sharp|=\Theta_n(|k|^{\rank G})$, and $d_m(G)\ll_n|k|^{\frac{m\dim(G)+2\rank(G)}{2}}$ for every $m\ge 3$.
  \item If $G=\SL_n(k)$ or $\Sp_n(k)$ ($n$ even), the upper bounds can be improved to $d_m(G)\ll_n |k|^{\frac{m\dim(G)+(2-m)\rank(G)}{2}}$ for every $m\ge 1$.
  \end{enumerate} 
\end{lemma}
\begin{proof}
  The relations for finite and finite abelian groups are well-known (see e.g. \cite[Proposition 5.2]{KowLargeSieve08}), the ones for classical groups follow from the former, and \cite[Corollary 24.6, Corollary 26.10]{TesMal11}, while the ones for $\SL_n$ and $\Sp_n$ are \cite[Proposition 5.4]{KowLargeSieve08}.
\end{proof}
\begin{remark}
  According to Remark \ref{rem:modelableFixG}, we do not keep track of implied constants depending on the rank of the monodromy group.
\end{remark}
\subsection{Conclusion}

\begin{proof}[Proof of Theorem \ref{thm:model}]
  By Proposition \ref{prop:EModel},
  \begin{eqnarray*}
    \E \left[h\Big((\rho_\Fc^\sharp(\Frob_{x+a}))_{a\in I}\Big)\right]&=&\E(h(Y_1,\dots,Y_L))\\
                                                                      &&+O \left(||h||_\infty\max_{\bs v\in (G^\sharp)^L} |E(\bs v)|\right)\nonumber
  \end{eqnarray*}
  By Proposition \ref{prop:sumProductsTranslatesFinite} (which applies by Propositions \ref{prop:finiteGKR} and \ref{prop:finiteGKRmud}) and Lemma \ref{lemma:sizesGEell},
  \begin{eqnarray*}
    \frac{\max_{\bs v\in (G^\sharp)^L} |E(\bs v)|\sqrt{q}}{\cond(\Fc)^2|G|^\delta}&\ll& \sum_{\chi_1,\dots,\chi_L\in \widehat G}\left(\prod_{i=1}^L \dim\eta_i\right)^2\sum_{i=1}^L \dim\eta_i\\
    &=&Ld_1(G)^{L-1}d_3(G).
  \end{eqnarray*}  
  The case $h$ nonnegative is treated similarly by the second part of Proposition \ref{prop:EModel}.
\end{proof}

\begin{proof}[Proof of Corollary \ref{cor:model}]
  By writing
  \begin{eqnarray*}
    \E \left[h\Big((t_\Fc(x+a))_{a\in I}\Big)\right]&=&\frac{q}{|U_{\Fc,I}(\F_q)|}\E \left[(h\circ\tr)\Big((\rho_\Fc^\sharp(\Frob_{x+a}))_{a\in I}\Big)\right]\\
                                                    &=&\E \left[(h\circ\tr)\Big((\rho_\Fc^\sharp(\Frob_{x+a}))_{a\in I}\Big)\right]\\
                                                    &&+ O \left(q^{-1}||h||_{\infty}L\cond(\Fc)\right),
  \end{eqnarray*}
  the first relation follows by Theorem \ref{thm:model}, and one argues similarly for the second one.
\end{proof}

\section{Computations in the model}\label{sec:compModel}
In this section, we carry out preliminary computations and observations in the probabilistic model.

Throughout, we let $G\le\GL_n(\F_\lf)$, $X_1,\dots,X_L$ independent random variables uniformly distributed in $G$, $Y_i=\pi(X_i)$ for $\pi: G\to G^\sharp$ the projection, and $Z_i=\tr Y_i$.

\subsection{Random walks in monodromy groups}

\begin{proposition}\label{prop:probModel}
  For all $A\subset\F_\lf$ and $L\ge 1$, the probability $P(Z_1+\dots+Z_L\in A)$ is given by
  \[\frac{|A|}{|\F_\lf|}+O \left(\max_{0\neq \psi\in\widehat\F_\lf}\left|\sum_{a\in A} \psi(-a)\right| \left|\frac{1}{|G|}\sum_{x\in G} \psi(\tr{x})\right|^L\right).\]
  In particular, for $a\in\F_\lf$,
  \[P(Z_1+\dots+Z_L=a)=\frac{1}{|\F_\lf|}+O \left(\max_{0\neq \psi\in\widehat\F_\lf} \left|\frac{1}{|G|}\sum_{x\in G} \psi(\tr{x})\right|^L\right).\]
\end{proposition}
\begin{proof}
  By Schur's orthogonality relations for the finite group $\F_\lf$,
  \begin{eqnarray*}
    P(Z_1+\dots+Z_L=a)&=&\frac{|\bs v=(v_1,\dots,v_L)\in G^L : \tr\sum v_i=a|}{|G|^L}\\
    &=&\frac{1}{|G|^L}\sum_{\bs v\in G^L} \delta_{\tr\sum v_i=a}\\
    &=&\frac{1}{|\F_\lf|}\sum_{\psi\in\widehat\F_\lf}\psi(-a)\left(\frac{1}{|G|}\sum_{v\in G}\psi(\tr{v})\right)^L\\
    &=&\frac{1}{|\F_\lf|}\left[1+\sum_{0\neq \psi\in\widehat\F_\lf}\psi(-a)\left(\frac{1}{|G|}\sum_{v\in G} \psi(\tr{v})\right)^L\right].
  \end{eqnarray*}
  The first statement follows from summing the previous equation over $a\in A$.
\end{proof}

\subsubsection{Gaussian sums}
For $\psi$ a nontrivial character of $\F_\lf$, we will call the sum $\sum_{v\in G} \psi(\tr{v})$ a ``Gaussian sum over $G$'', by analogy with the case $G=\mu_d(\F_\lf)$ (see Section \ref{subsubsec:GaussianSumsmud} below). We can expect it to be small uniformly with respect to $\psi$, say
\begin{equation}
  \label{eq:boundGaussSum}
  \frac{1}{|G|}\sum_{v\in G}\psi(\tr{v})\ll |\F_\lf|^{-\alpha(G)}
\end{equation}
with $\alpha(G)>0$, and square-root cancellation corresponds to $\alpha(G)\ge\frac{\log|G|}{2\log|\F_\lf|}$. Alternatively, we can also write
\begin{equation}
  \label{eq:boundGaussSum2}
  \sum_{v\in G}\psi(\tr{v})\ll |G|^{\alpha'(G)}\text{ with }\alpha'(G)<1.
\end{equation}

Similarly, if $A$ is ``well-distributed'' in $\F_\lf$, we expect
\begin{equation}
  \label{eq:boundGaussSumA}
  \frac{1}{|A|}\sum_{x\in A}\psi(-x)\ll|\F_\lf|^{-\alpha(A)}
\end{equation}
for some $\alpha(A)>0$, uniformly with respect to $\psi\in\widehat\F_\lf$. The trivial bound corresponds to $\alpha(A)=0$.

Thus, we can rewrite Proposition \ref{prop:probModel} as:

\begin{corollary}\label{cor:probModel}
  Let $A\subset\F_\lf$. If the bounds \eqref{eq:boundGaussSum} and \eqref{eq:boundGaussSumA} hold, then
  \[P(Z_1+\dots+Z_L\in A)=\frac{|A|}{|\F_\lf|}\left(1+O \left(\frac{1}{|\F_\lf|^{L\alpha(G)+\alpha(A)-1}}\right)\right)\]
  for all $L\ge 1$. In particular,
  \[P(Z_1+\dots+Z_L=a)=\frac{1}{|\F_\lf|}\left(1+O \left(\frac{1}{|\F_\lf|^{L\alpha(G)-1}}\right)\right)\]
  uniformly for all $a\in\F_\lf$.
\end{corollary}

It is insightful to distinguish the following cases to analyze the ranges of the parameters in Corollary \ref{cor:probModel}:
\begin{enumerate}
\item If either
  \begin{itemize}
  \item $\alpha(G)>1$, or
  \item $\alpha(G)\le 1$ and $L>1/\alpha(G)$,
  \end{itemize}
  we have equidistribution of $Z_1+\dots+Z_L$ in $\F_\lf$ as $|\F_\ell|\to+\infty$.
\item If $\alpha(G)\le 1$ and $L\le1/\alpha(G)$, then we have $P(Z_1+\dots+Z_L=a)\ll |\F_\lf|^{-L\alpha(G)}$, which shows that $Z_1+\dots+Z_L$ is ``not too concentrated'' at any point $a\in\F_\lf$.
\end{enumerate}

\begin{example}
  We will see that for $G=\SL_n(\F_\lf)$ or $\Sp_n(\F_\lf)$, we always have $\alpha(G)>1$. On the other hand, $\alpha(\mu_d(\F_\lf))<1$.
\end{example}

In the next two sections, we investigate bounds of the form \eqref{eq:boundGaussSum} (or \eqref{eq:boundGaussSum2}) for the monodromy groups $G$ we are interested in: roots of unity and classical groups over finite fields.

\subsection{Gaussian sums in $\mu_d(\F_\lf)$: exponential sums over subgroups of $\F_\lf^\times$}\label{subsubsec:GaussianSumsmud}

We assume that $\F_\lf$ contains a primitive $d$th root of unity. For $G=\mu_d(\F_\lf)\le\F_\lf^\times$, the sum \eqref{eq:boundGaussSum} is a ``character sum with exponentials''
\[\sum_{v\in\mu_d(\F_\lf)} \psi(v)=\sum_{i=1}^d \psi(\zeta_d^i),\]
or equivalently a sum over a subgroup of $\F_\lf^\times$.\\

For $\F_\lf=\F_\ell$, the latter appear in works of Korobov, Shparlinski, Heath-Brown-Konyagin, Konyagin, Bourgain-Glibichuk-Konyagin and others, which give nontrivial bounds for $d$ not too small compared to $\ell$. Square-root cancellation corresponds to $\alpha(G)\ge \frac{\log{d}}{2\log{\ell}}$, and $\frac{\log{d}}{\log{\ell}}<1$ since $\ell\equiv 1\pmod{d}$.

We first review the results of Heath-Brown-Konyagin which give explicit bounds for $d$ at least of the order of $\ell^{1/3}$.

\begin{theorem}[{\cite[Theorem 1]{HBKon00}}]\label{thm:HBKon}
  For $\psi$ a nontrivial additive character of $\F_\ell$, \eqref{eq:boundGaussSum} holds with $G=\mu_d(\F_\ell)$ and $\alpha(G)=\alpha$ in any of the following three cases:
  \begin{eqnarray}
    0<\alpha\le 1/16&\text{and}&d\gg\ell^{1/3+8\alpha/3}\\
      0<\alpha\le 1/6&\text{and}&d\gg\ell^{2/5+8\alpha/5}\\
      0<\alpha\le 1/2&\text{and}&d\gg\ell^{1/2+\alpha}.
  \end{eqnarray}
\end{theorem}

On the other hand, the results of Bourgain and others give (non-explicit) bounds for $d$ as small as desired:

\begin{theorem}[{\cite[Theorem 2.1]{BouKon03}}]\label{thm:Bourgain}
  Let $x,y\in\F_\ell^\times$ and let $d$ be the order of $y$. For every $\delta>0$, there exists $\alpha=\alpha(\delta)>0$ such that if $d\ge \ell^\delta$, then
  \[\sum_{i=1}^d \psi(y^ix)\ll d\ell^{-\alpha}\]
  uniformly for all nontrivial $\psi\in\widehat\F_\ell$, with an absolute implied constant. Thus, \eqref{eq:boundGaussSum} for $G=\mu_d(\F_\ell)$ holds with $\alpha(G)=\alpha(\delta)$ if $d\ge\ell^{\delta}$.
\end{theorem}

\begin{remark}
  The $\alpha(\delta)$ arising in Theorem \ref{thm:Bourgain} are not estimated explicitly in \cite{BouKon03}\footnote{This could be done with some effort using e.g. \cite{Garaev07} (see also \cite{Kow11ExplMultComb}).}, but one typically expects them to be very small.
\end{remark}

The situation is more complicated when $\F_\lf$ has nonprime order:

By using the formalism of trace functions (or the properties of general Artin-Schreier sheaves in the case of additive characters), we can get a result valid in the range of Korobov's:

\begin{proposition}
  Let $H$ be a subgroup of $\F_q^\times$ of index $k$ and $t: \F_q\to\C$ be a trace function corresponding to a geometrically irreducible $\ell$-adic sheaf $\Fc$ on $\P^1_{\F_q}$. If either $\rank(\Fc)>1$ or if the function $x\mapsto t(x^k)$ is nonconstant for $x^k\in U_\Fc(\F_q)$, then
  \[\sum_{x\in H} t(x)\ll \cond(\Fc)^2\sqrt{q}.\]
\end{proposition}
\begin{proof}
  Since $\F_q^\times$ is cyclic, we have $H=\{x^k : x\in\F_q^\times\}$ and $\sum_{x\in H} t(x)=\frac{1}{k}\sum_{x\in\F_q^\times} t(x^k)$. The sheaf $\Fc'=[x\mapsto x^k]^*\Fc$ is geometrically irreducible and $t_{\Fc'}(x)=t(x)$ when $x^k\in U_\Fc(\F_q)$. By Theorem \ref{thm:sumTF}, the sum is
  \[\ll\frac{\rank(\Fc')\cond(\Fc')}{k}\sqrt{q}\ll\frac{\rank(\Fc)k\cond(\Fc)}{k}\sqrt{q}\le\cond(\Fc)^2\sqrt{q},\]
  unless $\Fc'$ is geometrically trivial. In the latter case $\Fc'\cong\alpha\otimes\overline\Q_\ell$ for some $\alpha\in\overline\Q_\ell$ by Clifford theory (since $\pi_{1,q}/\pi_{1,q}^\geom\cong\Gal(\overline\F_q/\F_q)$), so that $t_{\Fc}(x^k)=\alpha$ whenever $x^k\in U_\Fc(\F_q)$.
\end{proof}
\begin{corollary}\label{cor:boundSubgroupFlambda}
  The bound \eqref{eq:boundGaussSum} for $G=\mu_d(\F_\lf)$ holds uniformly with respect to all nontrivial $\psi\in\widehat\F_\lf$ with $\alpha(G)=\alpha\in(0,1/2)$ whenever $d\ge |\F_\lf|^{1/2+\alpha}$.
\end{corollary}

\begin{remark}
  Alternatively, one could also proceed by completion as in \cite{Kor89}.
\end{remark}

By \cite{BouCh06}, the results of Bourgain and others (Theorem \ref{thm:Bourgain}) generalize to all finite fields, up to adding an assumption involving subfields:
\begin{theorem}[{\cite[Theorem 2]{BouCh06}}]
  For every $\delta>0$, there exists $\alpha=\alpha(\delta)>0$ such that with $\alpha(G)=\alpha$, \eqref{eq:boundGaussSum} for $G=\mu_d(\F_\lf)$ holds if
  \begin{equation}
    \label{eq:BouChCond}
    d>|\F_\lf|^\delta\text{ and }\frac{d}{(d,|F^\times|)}\ge|\F_\lf|^{\delta}
  \end{equation}
for all proper subfields $F\le\F_\lf$.
\end{theorem}

\begin{remark}\label{rem:BouChCond}
  Note that Condition \eqref{eq:BouChCond} amounts to $d>|\F_\lf|^{\delta}$ in the following situations:
  \begin{itemize}
  \item $d$ is prime and $\F_\lf=\F_\ell(\mu_d)$, or
  \item $\F_\lf=\F_\ell$ (recovering Theorem \ref{thm:Bourgain}) or $[\F_\lf:\F_\ell]$ is prime, or
  \item $\delta>1/2$ (recovering Corollary \ref{cor:boundSubgroupFlambda}).
  \end{itemize}
\end{remark}

\subsection{Gaussian sums in classical groups over finite fields}

Let us now assume that $G$ is a finite classical group of Lie type in $\GL_n(\F_\lf)$.

\begin{proposition}\label{prop:alphaG} Let $\F_\lf$ be a finite field and $n\ge 2$ be an integer. The bound \eqref{eq:boundGaussSum} holds for
  \begin{center}
    \begin{tabular}{c|c}
      $G$&$\alpha(G)$\\\hline
      $\GL_n(\F_\lf)$&$\frac{n(n-1)}{2}$\\
      $\SL_n(\F_\lf)$&$\frac{n^2-1}{2}$\\
      $\Sp_n(\F_\lf)$, $\SO_n^-(\F_\lf)$ \textup{(}$n$\textup{ even)}&$\frac{n(n+2)}{8}$\\
      $\SO_n(\F_\lf)$ \textup{(}$n$\textup{ odd)}&$\frac{n^2-1}{8}$\\
      $\SO_n^+(\F_\lf)$ \textup{(}$n$\textup{ even)}&$\frac{n(n-2)}{8}$
    \end{tabular}
  \end{center}
\end{proposition}
\begin{remark}
  By Lemma \ref{lemma:sizesGEell},
\[\frac{\log|G|}{\log|\F_\lf|}=\dim{G}+O_n \left(\frac{1}{\log{|\F_\lf|}}\right),\]
so square-root cancellation corresponds to $\alpha(G)>\dim(G)/2$. Hence, by the dimensions given in Table \ref{table:dimrank}, there is square-root cancellation in the special linear case, but not for the others. Note that the quality of the bounds improves as $n$ grows.
\end{remark}

\begin{proof}
  We use the explicit evaluation of Gaussian sums over finite classical groups carried out in \cite{Kim97}, \cite{Kim98SO} and \cite{Kim98} using the Bruhat decomposition. Let $a\in\F_\lf^\times$ corresponding to $\psi$ through the isomorphism $\widehat\F_\lf\cong\F_\lf$.
  \ifamsart
  \begin{enumerate}[leftmargin=*]
  \else
  \begin{enumerate}
  \fi
  \item By \cite[Theorem 4.2]{Kim97}, the Gaussian sum \eqref{eq:boundGaussSum} for $\GL_n(\F_\lf)$ is equal to $(-1)^n|\F_\lf|^{\frac{n(n-1)}{2}}$.
  \item By \cite[Corollary 5.2]{Kim97}, Deligne's bound for hyper-Kloosterman sums and Lemma \ref{lemma:sizesGEell}, the Gaussian sum \eqref{eq:boundGaussSum} for $\SL_n(\F_\lf)$ is
    \[\frac{|\F_\lf|^{\binom{n}{2}}\Kl_n(a^n)}{|G|}\ll_n|\F_\lf|^{\frac{n^2-n}{2}+\frac{n-1}{2}-n^2+1}=|\F_\lf|^{\frac{-n^2+1}{2}}.\]
  \item\label{item:KimSp} By \cite[Theorem A]{Kim98}, the Gaussian sum $\sum_{v\in \Sp_{2m}(\F_\lf)}\psi(\tr{v})$ is equal to
    \begin{eqnarray*}
      L^{m^2-1}\sum_{r=0}^{\floor{m/2}} L^{r(r+1)}\binom{m}{2r}_L\prod_{i=1}^r&& (L^{2i-1}-1)\\
      &&\times\sum_{l=1}^{\floor{m/2}-r+1} L^l \Kl_2(a^2)^{m-2r+2-2l}\\
      &&\times\sum_{j_1,\dots,j_{l-1}} (L^{j_1}-1)\dots (L^{j_{l-1}}-1)
    \end{eqnarray*}
for $L=|\F_\lf|$, where the last sum is over integers $2l-3\le j_1\le m-2r-1$, $2l-5\le j_2\le j_1-2$, \dots, $1\le j_{l-1}\le j_{l-2}-2$ and
\[\binom{m}{r}_L=\prod_{j=0}^{r-1} \frac{L^{m-j}-1}{L^{r-j}-1}\ll_m L^{r(m-r)}.\]
Using that $\prod_{i=1}^r (L^{2i-1}-1)<L^{r^2}$ and
\[\Kl_2(a^2)^{t+2-2l}\sum_{j_1,\dots,j_{l-1}} (L^{j_1}-1)\dots (L^{j_{l-1}}-1)\ll_nL^{(l-1)(t-(l-1))}\]
for $t=m-2r$ (see \cite[Remark (1) p. 65]{Kim98} for the second one), we find that the Gaussian sum is
\[\ll_m\begin{cases}
  |\F_\lf|^{\frac{3m^2+m}{2}}&\text{if }m\text{ even}\\
  |\F_\lf|^{\frac{2m^2+m-1}{2}}&\text{if }m\text{ odd},
\end{cases}\]
and the result follows by Lemma \ref{lemma:sizesGEell}.
\item By \cite[Theorem A]{Kim98SO},
  \[\sum_{v\in \SO_{2m+1}(\F_\lf)}\psi(\tr{v})=\psi(1)\sum_{v\in \Sp_{2m}(\F_\lf)}\psi(\tr{v}),\]
  the result follows by the previous bound and Lemma \ref{lemma:sizesGEell}.
\item Similarly, by \cite[Theorem 4.3]{KimLee98},
  \[\sum_{v\in \SO_{2m}^+(\F_\lf)}\psi(\tr{v})=|\F_\lf|^{-m}\sum_{v\in \Sp_{2m}(\F_\lf)}\psi(\tr{v}).\]
\item This is analogous to \ref{item:KimSp}, using \cite[Theorem A]{Kim97Om}.
\end{enumerate}
\end{proof}

\section{Equidistribution of shifted short sums}\label{sec:EDSSS}

As a first application of Theorem \ref{thm:model}, we prove in particular Propositions \ref{prop:EDShiftsKl} and \ref{prop:EDShiftsChi} introduced in Section \ref{subsubsec:introEDSSS}.

\subsection{Statement of the result}

\begin{theorem}\label{thm:equidistribution}
  Let $(\Fc_\lambda)_{\lambda\in\Lambda}$ be a coherent family with monodromy group structure $G$. For $\lambda\in\Lambda$, let $t: \F_q\to\F_\lf$ be the trace function associated to the sheaf $\Fc=\Fc_\lambda$, $I\subset\F_q$ of size $L$ such that $\Fc$ is $I$-compatible, and $a\in\F_\lf$.
  \ifamsart
  \begin{enumerate}[leftmargin=*,itemsep=0.2cm]
  \else
  \begin{enumerate}
  \fi
\item If $G$ is classical, then the probability $P\big(S(t,I+x) \equiv a\big)$ (with respect to the uniform measure on $\F_q$) is equal to
  \begin{equation}
    \label{eq:probaSLSp}
    \frac{1}{|\F_\lf|}+O\left(\frac{1}{|\F_\lf|^{L\alpha(G)}}+\frac{L|\F_\lf|^{L\beta_+(G)+2\beta_-(G)-1}}{q^{1/2}}\right),
    \end{equation}
    uniformly with respect to $a$, where $\alpha(G),\beta_\pm(G)>0$ are given explicitly in Table \ref{table:dimrank}.
  \item If $G=\mu_d$, for every $\delta\in(0,1)$ there exists $\alpha=\alpha(\delta)>0$ such that the probability the probability $P\big(S(t,I+x) \equiv a\big)$ is
    \begin{equation}
      \label{eq:probamud}
      \frac{1}{|\F_\lf|}+O\left(\frac{1}{|\F_\lf|^{L\alpha}}+\frac{Ld^{L+1}}{q^{1/2}|\F_\lf|^{\min(L\alpha,1)}}\right)
    \end{equation}
    uniformly with respect to $a$, when Condition \eqref{eq:BouChCond} holds\footnote{See also Remark \ref{rem:BouChCond}.} for all proper subfields $F\le\F_\lf$. Moreover:
    \begin{itemize}
    \item If $\delta>1/2$, we can choose $\alpha(\delta)=\delta-1/2$. If $d$ is prime, the factor $d^{L+1}$ can be replaced by $d^L$.
    \item If $\F_\lf=\F_\ell$, then Condition \eqref{eq:BouChCond} is $d\ge\ell^\delta$ and explicitly, we can choose $\alpha(\delta)$ as in \eqref{eq:explicitalpha}.
    \end{itemize}
  \end{enumerate}
  The implied constants depend only on the monodromy group structure and on a bound on the conductor of the family.
\end{theorem}
\begin{proof}
  By Corollary \ref{cor:model} and Proposition \ref{prop:probModel}, we have for all $a\in\F_\lf$ that the probability $P\left(S(t,I+x)\equiv a\right)$ is equal to
\begin{equation}
  \frac{1}{|\F_\lf|}+O \left(\frac{1}{|\F_\lf|^{L\alpha(G)}}+\frac{LE(G,L,\F_\lf)}{q^{1/2}|\F_\lf|^{\min(L\alpha(G),1)}}\right)
\end{equation}
In the classical case, note that $\alpha(G)>1$ by Table \ref{table:dimrank}.
\end{proof}

\subsection{Analysis of the ranges}
\subsubsection{Case $G$ classical} Since $\alpha(G)>1$, the error term of \eqref{eq:probaSLSp} is negligible with respect to the main term (i.e. with a ratio which is $o(1)$) when $L|\F_\lf|^{L\beta_+(G)+2\beta_-(G)}=o(q^{1/2})$. Note that:
\begin{itemize}
\item When $L=1$, this is $|\F_\lf|=o\left(q^{1/\dim(G)}\right)$.
\item Let $q=p^e$. When $d=p$ (e.g. for Kloosterman sums), this implies that $e>2(L\beta_+(G)+2\beta_-(G))$ (see Section \ref{subsec:limitationRange}).
\end{itemize}
\subsubsection{Case $G$ cyclic} The error term of \eqref{eq:probaSLSp} is negligible with respect to the main term when $L>1/\alpha>1$ and $Ld^{L+1}=o(q^{1/2})$. In particular, $1<1/\alpha<L<\log{q}/2$.
\subsection{Examples}\label{subsec:equidistributionRedExamples}

\subsubsection{Kloosterman sums}
By Proposition \ref{prop:KlCoherent}, Theorem \ref{thm:equidistribution} gives Proposition \ref{prop:EDShiftsKl}. Replacing $a$ by $aq^{(n-1)/2}$ and using the uniformity statement shows that the results hold as well for unnormalized Kloosterman sums.

\subsubsection{Point-counting on families of curves}
The case $n$ even of Proposition \ref{prop:EDShiftsKl} also applies to the point-counting on families of hyperelliptic curves from Proposition \ref{prop:hyperellipticCoherent} when $\F_\lf=\F_\ell$, normalized or not.

\subsubsection{Multiplicative characters} By Proposition \ref{prop:KummerCoherent} and Example \ref{ex:noZeromSum}, Theorem \ref{thm:equidistribution} yields Proposition \ref{prop:EDShiftsChi}.

\section{Distribution of families of short sums}\label{sec:DistrFam}

As a second application of the probabilistic model developed above, we generalize the results of \cite{LamzModm} on the distribution of residues of sums over partial intervals of the Legendre symbol to the distribution of sums of reduced trace functions in coherent families, giving in particular the results from Section \ref{subsubsec:introDistrFam}.


\subsection{Families of short sums}

\subsubsection{Definition and examples}

\begin{definition}
  Let $t: \F_q\to\F_\lf$ be any function. A \textit{family of sums with respect to $t$} is a family
  \begin{equation}
    \label{eq:familyShortSums}
    \Big(S(t,\Ic(k))\Big)_{k\in\Ic}
  \end{equation}
  for a finite parameter space $\Ic$ with an injective map $\Ic\to\Pc(\F_q)$, $k\mapsto \Ic(k)$, where $\Pc(\F_q)$ is the set of subsets of $\F_q$.
\end{definition}

\begin{examples}\label{ex:familiesShortSums}~
  \begin{enumerate}
  \item \label{ex:familiesShortSumsIntervals} (Intervals) When $q=p$, we can study sums over the integer intervals $\{\Ic(k)=\{1,\dots, k\} : k\in\Ic\}$ for a parameter set $\Ic\subset\{1,\dots,p\}\cong\F_p$.
  \item \label{ex:familiesShortSumsBoxes} (Boxes) More generally, when $q=p^e$, we can fix a $\F_p$-basis of $\F_q$, identify $\F_q$ with $\{1,\dots,p\}^e$, and study sums over the ``boxes''\footnote{Of course, one should not replace the sums over $\{\{1,\dots, k\} : 1\le k\le p\}$ by the sums over $\{\{1,\dots, k\} : 1\le k\le q\}$.}
    \[\Ic(k)=\{1,\dots, k_1\}\times\{1,\dots, k_2\}\times\dots\times\{1,\dots, k_{e}\},\]
    with $k=(k_1,\dots,k_e)\in\Ic\subset\{1,\dots,p\}^e$.
  \item \label{ex:familiesShortSumsShifts} (Shifted subsets) For $\Ic,E\subset\F_q$, we can consider the translates $\Ic(x)=E+x=\{y+x : y\in E\}$ of $E$ by elements $x\in \Ic$.
  \item (Combining families) \label{ex:familiesShortSumsCombining} Given families $\Ic_i\to\Pc(\F_p)$ ($i=1,\dots,e$), we can form the family $\Ic=\Ic_1\times\dots\times\Ic_e$ over $\F_q\cong\F_p^e$ defined by
    \[\Ic(k_1,\dots,k_e)=\prod_{i=1}^e \Ic_i(k_i)\subset\F_q.\]
  \end{enumerate}
\end{examples}

\subsubsection{Distribution questions}

We are interested in the distribution of the random variable \eqref{eq:familyShortSums} with the uniform measure on $\Ic$, asymptotically with respect to the parameters $q$ and $|\F_\lf|$. Thus, we are led to study the density
\[\Phi(t,\Ic,a):=\frac{|\{k\in\Ic : S(t,\Ic(k))\equiv a\}|}{|\Ic|} \ \ (a\in\F_\lf).\]

\begin{example}\label{ex:LamzModm}
  Let $\ell\ge 2$ be an integer and consider the family $\Ic$ of Example \ref{ex:familiesShortSums} \ref{ex:familiesShortSumsIntervals} with $t=\legendre{\cdot}{p}: \F_p\to\F_\ell$ the Legendre symbol, a multiplicative character of order $2$. As we mentioned in the introduction, one of the main results of \cite{LamzModm} is that
  \[\Phi(t,\Ic,a)=\frac{1}{\ell}+ O \left(\left(\frac{\ell}{\log{p}}\right)^{\frac{1}{2}}\right)\]
  uniformly with respect to $a\in\F_\ell$, as $p\to+\infty$. Therefore, the random variable \eqref{eq:familyShortSums} converges in law to the uniform distribution on $\F_\ell$ if $\ell$ is fixed, $p\to+\infty$, and more generally we have $\Phi(t,\Ic,a)\sim \frac{1}{\ell}$ if $\ell=o((\log{p})^{1/3})$.
\end{example}

Our goal is to generalize this result in different directions: for other reductions of trace functions (such as multiplicative characters of any order, Kloosterman sums and point-counting functions on families of curves), for other families of short sums, and in the case $q>p$.

\begin{example}
  The study of $\Phi(t,\Ic,a)$ for the family of Example \ref{ex:familiesShortSums} \ref{ex:familiesShortSumsShifts} is the finite analogue of the distribution questions considered in \cite{PGGaussDistr16}, generalizing \cite{LamzShortSums} to trace functions.
\end{example}

\subsection{Equidistribution on average/for shifted families}

Given a rather generic family $\Ic\to\Pc(\F_q)$, we could expect the random variable \eqref{eq:familyShortSums} to converge to the uniform distribution on $\F_\lf$. Albeit we cannot show that in this most general setting, we have nonetheless a result on average over shifts.

\begin{definition}
  For a family $\Ic\to\Pc(\F_q)$, we denote by $\Ic'=\Ic\times\F_q\to\Pc(\F_q)$ the shifted family defined by $\Ic'(k,x)=\Ic(k)+x$, and we let the families
\begin{align*}
  \Ic+x&=\Ic'(\,\cdot\,,x):&& \Ic\to\Pc(\F_q)&&\text{ for }x\in\F_q,\\
  \Ic_k&=\Ic'(k,\,\cdot\,):&& \F_q\to\Pc(\F_q)&&\text{ for }k\in\Ic.
\end{align*}
\end{definition}

Hence, for a family $\Ic\to\Pc(\F_q)$, we have $\Ic=\Ic+0=\Ic'(\,\cdot\,,0)$ and
\begin{eqnarray}
  \Phi(t,\Ic',a)&=&\aver{x}{\F_q}{q} \Phi(t,\Ic+x,a)=\averC{k}{\Ic} \Phi(t,\Ic_k,a),\label{eq:exchangeSummation}\\
  \Phi(t,\Ic+x,a)&=&\frac{|\{k\in\Ic : S(t,\Ic(k)+x)\equiv a\}|}{|\Ic|} \hspace{0.4cm} (x\in\F_q),\nonumber\\
  \Phi(t,\Ic_k,a)&=&\frac{|\{x\in\F_q : S(t,\Ic(k)+x)\equiv a\}|}{q} \hspace{0.4cm} (k\in\Ic),\nonumber
\end{eqnarray}
for any function $t: \F_q\to\F_\lf$ and $a\in\F_\lf$.

\begin{definition}\label{def:constantsIc}
  For a family $\Ic\to\Pc(\F_q)$, we define the quantities
  \begin{eqnarray*}
    M_\Ic&=&\Big|\bigcup_{k\in\Ic} \Ic(k)\Big|,\\
    h_\Ic(d)&=&\Big|\{(k_1,k_2)\in\Ic^2 : |\Ic(k_1)\Delta\Ic(k_2)|=d\}\Big| \ (d\ge 1),\\
    H_\Ic(\alpha,n)&=&\frac{1}{|\Ic|}\sum_{d\ge 1} \frac{h_\Ic(d)}{n^{\alpha d}} \ (n>0, \ \alpha>0),
  \end{eqnarray*}
  where $\Delta$ denotes the symmetric difference operator.
\end{definition}

The following will be proven in Sections \ref{subsec:distrShortSumsProbModel}--\ref{subsec:approximateVariance}:
\begin{theorem}\label{thm:familiesShortSums}
   Let $(\Fc_\lambda)_{\lambda\in\Lambda}$ be a coherent family with monodromy group structure $G$. For $\lambda\in\Lambda$, let $t: \F_q\to\F_\lf$ be the trace function associated to the sheaf $\Fc=\Fc_\lambda$, and let $\Ic$ be a family of sums with respect to $t$ so that $\Fc$ is $\bigcup_{k\in\Ic} \Ic(k)$-compatible. The averaged variance
  \begin{equation}
    \label{eq:averagedVariance}
    V(t,\Ic)=\sum_{a\in\F_\lf}\frac{1}{q}\sum_{x\in\F_q} \left(\Phi(t,\Ic+x,a)-\frac{1}{|\F_\lf|}\right)^2
  \end{equation}
  is equal to
  \[\frac{1}{|\Ic|}\left(1+O\left(\tilde V(t,\Ic)\right)\right)\]
  with $\tilde V(t,\Ic)$ given by
  \[H_\Ic(\alpha(G),|\F_\lf|)\left(1+\frac{M_\Ic}{q^{1/2}}\times
        \begin{cases}
          |\F_\lf|^{\beta_+(G)M_\Ic+2\beta_-(G)}&\text{if } G\text{ classical}\\
          d^{M_\Ic+1}&\text{if } G\text{ cyclic}
        \end{cases}
\right),\]
for $\alpha(G), \beta_\pm(G)>0$ given in Table \ref{table:dimrank}. The implied constants depend only on the monodromy group structure and on a bound on the conductor of the family.
\end{theorem}

Thus, $V(t,\Ic)$ should be small as $|\Ic|\to+\infty$, and we have:
\begin{corollary}\label{cor:familiesShortSums}
  In the setting of Theorem \ref{thm:familiesShortSums},
  \begin{equation*}
  \Phi(t,\Ic',a)=\frac{1}{|\F_\lf|}+O\left(V(t,\Ic)^{1/2}\right)
\end{equation*}
uniformly with respect to $a\in\F_\lf$.
\end{corollary}
\begin{proof}
  By the Cauchy-Schwarz inequality applied to the sum over $x$,
  \begin{eqnarray*}
    \sum_{a\in\F_\lf}\left(\Phi(t,\Ic',a)-\frac{1}{|\F_\lf|}\right)^2&=&\sum_{a\in\F_\lf}\left(\aver{x}{\F_q}{q} \left(\Phi(t,\Ic+x,a)-\frac{1}{|\F_\lf|}\right)\right)^2\\
                                                                             &\le&V(t,\Ic).
  \end{eqnarray*}
\end{proof}

\subsection{Consequences}

Using Corollary \ref{cor:familiesShortSums}, we can obtain results for unshifted ``complete'' (i.e. parametrized by $\F_q$) families by averaging over an auxiliary family of appropriate size. This is the idea exploited in \cite{LamzModm} for the family of Example \ref{ex:LamzModm}.

The results presented in this section will be proven in Sections \ref{subsec:analysisErrorTerm}--\ref{subsec:applicationRemoveShifts}.

\begin{definition}
  For a subset $E\subset\F_q$ and a fixed choice of a $\F_p$-basis of $\F_q$ which identifies the latter with $\{1,\dots,p\}^e$, the \textit{bounding box} $B_E$ of $E$ is defined as $E\subset B_E=\prod_{i=1}^e [\min_{x\in E} x_i,\max_{x\in E} x_i]\subset\F_q$.
\end{definition}

\subsubsection{Shifts of small subsets}

First, we consider shifts of subsets of moderate size following Example \ref{ex:familiesShortSums} \ref{ex:familiesShortSumsShifts}. The Gaussian distribution for complex-valued trace functions from \cite{PGGaussDistr16} becomes a uniform distribution when the latter are reduced in $\F_\lf$:

\begin{proposition}[Shifts of small subsets]\label{prop:EDfixedSubsets}
  Let $(\Fc_\lambda)_{\lambda\in\Lambda}$ be a coherent family with monodromy group structure $G$. For $\lambda\in\Lambda$, let $t: \F_q\to\F_\lf$ be the trace function associated to the sheaf $\Fc=\Fc_\lambda$ on $\P^1_{\F_q}$. Let $\varepsilon,\varepsilon'\in(0,1/2)$, $\delta\in(0,1)$, and let $E\subset\F_q$. We assume that:
  \begin{itemize}
  \item For a fixed $\F_p$-basis of $\F_q$, identifying the latter with $\{1,\dots, p\}^e$, we have $|B_E|<q^{1/2-\varepsilon'}$ and $B_E\subset [0,\delta p)^e$.
  \item If $\Fc=\Lc_{\chi(f)}$ is a Kummer sheaf with $f=f_1/f_2\in\F_q(X)$, $(f_1,f_2)=1$, then $\delta<1/\deg(f_1)$.
  \item If $G=\mu_d$ with $\F_\lf\neq\F_\ell$ and $d$ is nonprime, then Condition \eqref{eq:BouChCond} holds \textup{(}e.g. $d\ge|\F_\lf|^{1/2+\alpha}$ for some $\alpha>0$\textup{)}.
  \end{itemize}
  Then the density
  \[P\big(S(t,E+x)\equiv a\big)=\frac{|\{x\in\F_q : S(t,E+x)\equiv a\}|}{q}\]
  is given by
  \[\frac{1}{|\F_\lf|}+
  \begin{cases}
    O \left(\frac{1}{q^{1/4-\varepsilon/2}}+ \left(\frac{|E|\log{|\F_\lf|}}{\log{q}}\right)^{\frac{1}{2}}\right)&\text{if } G\text{ classical}\\
    O \left(\frac{1}{q^{1/4-\varepsilon/2}}+ \left(\frac{|E|\log{d}}{\log{q}}\right)^{\frac{1}{2}}\right)&\text{if } G\text{ cyclic}
  \end{cases}
\]
  uniformly for all $a\in\F_\lf$, where the implied constants depend on $\varepsilon$, $\varepsilon'$, $\delta$, on the monodromy structure and on a bound on the conductor of the family.
\end{proposition}

\begin{remark}\label{rem:EDfixedSubsetsRange}
  This is nontrivial if
  \[|E|\log{|\F_\lf|}=o(\log{q})\hspace{0.2cm}\text{ (resp. }|E|\log{d}=o(\log{q})).\]
  Note that when the sheaf $\Fc$ of $\F_\lf$-modules from which $t$ arises comes from the reduction of a sheaf of $\Z[\zeta_{4p}]_\lambda$-modules (e.g. for Kloosterman sums), we must thus take $|E|=o(e)$ (see Section \ref{subsec:limitationRange}).
\end{remark}

\begin{remark}
The first condition about $B_E$ in the statement can be included in the second one by taking $\delta< p^{-1/2-\varepsilon'}$.
\end{remark}

By taking $E=\{0\}$, we get the following corollary, which should be compared with the case $I=\{0\}$ of Theorem \ref{thm:equidistribution}:

\begin{corollary}\label{cor:EDfixedSubsets}
Let $(\Fc_\lambda)_{\lambda\in\Lambda}$ be a coherent family with monodromy group structure $G$. For $\lambda\in\Lambda$, let $t: \F_q\to\F_\lf$ be the trace function associated to the sheaf $\Fc=\Fc_\lambda$ on $\P^1_{\F_q}$, and let $\varepsilon\in(0,1/2)$. We assume that if $G=\mu_d$ with $\F_\lf\neq\F_\ell$ and $d$ nonprime, then Condition \eqref{eq:BouChCond} holds \textup{(}e.g. $d\ge|\F_\lf|^{1/2+\alpha}$ for some $\alpha>0$\textup{)}. Then
  \[P\big(t(x)\equiv a\big)=\frac{1}{|\F_\lf|}+
  \begin{cases}
    O \left(\frac{1}{q^{1/4-\varepsilon/2}}+ \left(\frac{\log{|\F_\lf|}}{\log{q}}\right)^{\frac{1}{2}}\right)&: G\text{ classical}\\
    O \left(\frac{1}{q^{1/4-\varepsilon/2}}+ \left(\frac{\log{d}}{\log{q}}\right)^{\frac{1}{2}}\right)&: G\text{ cyclic}
  \end{cases}
\]
  uniformly for all $a\in\F_\lf$, where the implied constants depend on $\varepsilon$, on a bound on the conductor of the family, and on the type of $G$ in the classical case.
\end{corollary}

\begin{example}
  By Section \ref{subsubsec:examplesCoherent}, Proposition \ref{prop:EDfixedSubsets} and Corollary \ref{cor:EDfixedSubsets} apply to:
  \begin{itemize}
  \item Kloosterman sums of fixed rank (normalized or not) and multiplicative characters composed with rational functions, giving Proposition \ref{prop:shiftsSmallSubsets}.
  \item Point-counting functions for families of hyperelliptic curves (normalized or not).
  \end{itemize}
\end{example}

\subsubsection{Partial intervals}
The second example notably generalizes the result of \cite{LamzModm} (see Example \ref{ex:LamzModm}) to all multiplicative characters:

\begin{proposition}[Partial intervals]\label{prop:partialIntervals}
  Let $(\Fc_\lambda)_{\lambda\in\Lambda}$ be a coherent family with monodromy group structure $G$. For $\lambda\in\Lambda$, let $t: \F_q\to\F_\lf$ be the trace function associated to the sheaf $\Fc=\Fc_\lambda$, and let $\varepsilon,\varepsilon'\in(0,1/2)$. We assume that if $G=\mu_d$ with $\F_\lf\neq\F_\ell$ and $d$ nonprime, then Condition \eqref{eq:BouChCond} holds (e.g. $d\ge|\F_\lf|^{1/2+\alpha}$ for some $\alpha>0$). Then the density
  \[P\big(S(t,\{1,\dots, x\})\equiv a\big)=\frac{|\{1\le k\le p : S(t,\{1,\dots, k\})\equiv a\}|}{p}\]
  is given by
  \[\frac{1}{|\F_\lf|}+
  \begin{cases}
    O \left(\frac{1}{p^{1/4-\varepsilon/2}}+\left(\frac{\log{|\F_\lf|}}{\log{p}}\right)^{\frac{1}{2}}+\delta_{S(t,\F_p)\neq 0} \left(\frac{|\F_\lf|\log{p}}{p\log{|\F_\lf|}}\right)^{\frac{1}{2}}
\right)&\text{if }G\text{ classical}\\
    O \left(\frac{1}{p^{1/4-\varepsilon/2}}+\left(\frac{\log{d}}{\log{p}}\right)^{\frac{1}{2}}+\delta_{S(t,\F_p)\neq 0} \left(\frac{|\F_\lf|\log{p}}{p\log{d}}\right)^{\frac{1}{2}}
\right)&\text{if }G\text{ cyclic}
  \end{cases}
\]
  uniformly for all $a\in\F_\lf$, where the implied constants depend on $\varepsilon,\varepsilon'$, on a bound on the conductor of the family, and on the monodromy group structure of the family.
\end{proposition}

\begin{examples}
  By Section \ref{subsubsec:examplesCoherent}:
  \begin{enumerate}
  \item This applies to multiplicative characters of $\F_p^\times$ of order $d$ composed with $f\in\Q(X)$ whose zeros and poles have orders not divisible by $d$, giving Proposition \ref{prop:EDShiftsChi}. When $\chi$ is the Legendre symbol, this is the result of \cite{LamzModm}. By the orthogonality relations, the third summand of the error term vanishes if $f=X$.
  \item With $d=2$ and $\F_\lf=\F_\ell$, this also applies to the point-counting functions on families of hyperelliptic curves from Proposition \ref{prop:hyperEllipticFamily}. See also \cite{MakZahar14} for an analogue of \cite{LamzModm} to the counting of points of a plane curve in rectangles.
  \end{enumerate}
\end{examples}

\begin{remark}
  We will see that it is unclear whether this can be generalized to the case $e\ge 2$ (see Example \ref{ex:familiesShortSums} \ref{ex:familiesShortSumsBoxes}) because of ``diagonal'' terms in the errors. Since the case $d=p$, $G$ classical forces to take $e\to+\infty$ (see Section \ref{subsec:limitationRange} and Remark \ref{rem:EDfixedSubsetsRange}), Proposition \ref{prop:partialIntervals} does not make sense for Kloosterman sums.
\end{remark}
Even though Proposition \ref{prop:partialIntervals} does not extend to ``boxes'' in $\F_q\cong\F_p^e$ with $e\ge 2$, we nonetheless have the following for a family of type \ref{ex:familiesShortSumsCombining} from Example \ref{ex:familiesShortSums}.

\begin{proposition}[Partial intervals with shifts of small subsets]\label{prop:partialIntervalShifts}
  Let $(\Fc_\lambda)_{\lambda\in\Lambda}$ be a coherent family with monodromy group structure $G$. For $\lambda\in\Lambda$, let $t: \F_q\to\F_\lf$ be the trace function associated to the sheaf $\Fc=\Fc_\lambda$ on $\P^1_{\F_q}$, and let $\varepsilon,\varepsilon'\in(0,1/2)$, $\delta\in(0,1)$. We fix a $\F_p$-basis of $\F_q$ and identify the latter with $\{1,\dots,p\}^e$. We let $E_2,\dots,E_e\subset\{1,\dots,p\}$ be such that
  \[|B_{E}|\le q^{1/2-\varepsilon'}\text{ and }E_i\subset[1,\delta p) \ (2\le i\le e),\]
  where $B_E$ is the bounding box of $E=E_2\times\dots\times E_e$ in $\F_p^{e-1}$. Moreover, we assume that:
  \begin{itemize}
  \item If $\Fc$ is a Kummer sheaf $\Lc_{\chi(f_1/f_2)}$, then $\delta<1/\deg(f_1)$.
  \item If $G=\mu_d$ and $\F_\lf\neq\F_\ell$ and $d$ is nonprime, then Condition \eqref{eq:BouChCond} holds (e.g.  if $d\ge|\F_\lf|^{1/2+\alpha}$ for some $\alpha>0$).
  \end{itemize}
  Then the density
  \[\frac{|\{(x_1,\dots,x_e)\in\F_p^e : S(t,\{1,\dots,x_1\}\times\prod_{i=2}^e (E_i+x_i))\equiv a\}|}{q}\]
  \textup{(}with respect to any $\F_p$-basis of $\F_q$\textup{)} is equal to
  \[\frac{1}{|\F_\lf|}+
  \begin{cases}
    O \left(\frac{1}{q^{1/4-\varepsilon/2}}+ \left(\frac{|E|\log{|\F_\lf|}}{\log{q}}\right)^{\frac{1}{2}}\right)&\text{if } G\text{ classical}\\
    O \left(\frac{1}{q^{1/4-\varepsilon/2}}+ \left(\frac{|E|\log{d}}{\log{q}}\right)^{\frac{1}{2}}\right)&\text{if } G\text{ cyclic}
  \end{cases}
\]
  uniformly for all $a\in\F_\lf$, where the implied constants depend on $\varepsilon$, $\varepsilon'$ and $\delta$.
\end{proposition}

\begin{example}
  As for Proposition \ref{prop:EDfixedSubsets}, this applies by Section \ref{subsubsec:examplesCoherent} to:
  \begin{itemize}
  \item Kloosterman sums of fixed rank (normalized or not), multiplicative characters composed with rational functions, giving Proposition \ref{prop:partialIntervalShifts1}.
  \item Point-counting functions on families of hyperelliptic curves (normalized or not).
  \end{itemize}  
\end{example}

\subsection{Probabilistic model}\label{subsec:distrShortSumsProbModel}

Let $\Fc$ be a sheaf of $\F_\lf$-modules on $\P^1_{\F_q}$, part of a coherent family, with monodromy group $G\le\GL_n(\F_\lf)$. We apply the probabilistic model from Section \ref{sec:model} to study of the distribution of families of short sums.

Again, we let $X$ be a random variable uniformly distributed in $G$, and $Z$ be its image through the map $G\to G^\sharp\xrightarrow{\tr}\F_\lf$. Moreover, let $(Z_i)_{i\in\N}$ be a sequence of independent random variables distributed like $Z$.

For a finite subset $I\subset\N$, we define the random variable
\[S(I)=\sum_{i\in I} Z_i\]
on the probability space $G^\N$. For a finite parameter space $\Ic$ with a map $\Ic\to\Pc_f(\N)$, we consider for all $a\in\F_\lf$ the random variable
\[\Phi(\Ic,a)=\frac{|\{k\in\Ic : S(\Ic(k))\equiv a\}|}{|\Ic|}.\]

In this setting, Corollary \ref{cor:model} gives information about the distribution of $\Phi(t,\Ic,a)$ averaged over shifts of the family $\Ic$ by elements of $\F_q$:

\begin{proposition}\label{prop:modelFamiliesSums}
  In the above setting, if $\Fc$ is $\bigcup_{k\in\Ic} \Ic(k)$-compatible, for any function $h: \F_\lf\to\R_{\ge 0}$ and any $a\in\F_\lf$, we have
  \[\E\Big(h(\Phi(t,\Ic+x,a))\Big)=\E\Big(h(\Phi(\Ic,a))\Big)\left(1+O \left(\frac{M_\Ic  E(G,M_\Ic,\F_\lf)}{q^{1/2}}\right)\right).\]
\end{proposition}

In other words, for all $a\in\F_\lf$ the random variable $(\Phi(t,\Ic+x,a))_{x\in\F_q}$ converges in law (with respect to the parameters, $q$, $|\F_\lf|$, $\Ic$) to the random variable $\Phi(\Ic,a)$ if the error term is $o(1)$ as the parameters vary.

\subsection{Expected value}
  
We first consider the expected value of $\Phi(\Ic,a)$, which gives a preliminary version of Theorem \ref{thm:familiesShortSums} and a motivation for the next section, where the former will be improved by analyzing the variance. The improvement will concern the quality of the error term, the uniformity with respect to $a$, and the ability to obtain Proposition \ref{prop:partialIntervals} by removing the shifts for some specific families.

\subsubsection{Computation in the model}

\begin{definition}\label{def:constantsIc2}
  For a family $\Ic\to\Pc(\F_q)$, we define
  \begin{eqnarray*}
    g_\Ic(d)&=&|\{k\in\Ic : |\Ic(k)|=d\}| \ (d\ge 0),\\
    G_\Ic(\alpha,n)&=&\frac{1}{|\Ic|}\sum_{d\ge 1} \frac{g_\Ic(d)}{n^{\alpha d}} \ (n>0, \ \alpha>0).
  \end{eqnarray*}
\end{definition}

\begin{proposition}\label{prop:PSIc}
  In the notations of Section \ref{subsec:distrShortSumsProbModel}, we have for $a\in\F_\lf$:
  \[\E\left(\Phi(\Ic,a)\right)=\frac{1}{|\F_\lf|}+O \left(G_\Ic(\alpha(G),|\F_\lf|)\right).\]
\end{proposition}
\begin{proof}
  By Corollary \ref{cor:probModel},
  \begin{eqnarray*}
    \E\left(\Phi(\Ic,a)\right)&=&\frac{1}{|\Ic|}\sum_{k\in\Ic} \E(\delta_{S(\Ic(k))\equiv a})=\frac{1}{|\Ic|}\sum_{k\in\Ic} P(S(\Ic(k))\equiv a)\\
                              &=&\frac{1}{|\Ic|}\sum_{k\in\Ic} \left(\frac{1}{|\F_\lf|}+O \left(|\F_\lf|^{-|\Ic(k)|\alpha(G)}\right)\right)\\
                              &=&\frac{1}{|\F_\lf|}+ O \left(\frac{1}{|\Ic|}\sum_{k\in\Ic} |\F_\lf|^{-|\Ic(k)|\alpha(G)}\right).
  \end{eqnarray*}
\end{proof}

\subsubsection{Conclusion} By Propositions \ref{prop:modelFamiliesSums} and \ref{prop:PSIc}, we get the following preliminary version of Theorem \ref{thm:familiesShortSums}:

\begin{proposition}\label{prop:familiesShortSums}
  Let $(\Fc_\lambda)_{\lambda\in\Lambda}$ be a coherent family with monodromy group structure $G$. For $\lambda\in\Lambda$, let $t: \F_q\to\F_\lf$ be the trace function associated to the sheaf $\Fc=\Fc_\lambda$ on $\P^1_{\F_q}$, and let $\Ic$ be a family of sums such that $\Fc$ is $\bigcup_{k\in\Ic} \Ic(k)$-compatible. For all $a\in\F_\lf$,
  \[\E(\Phi(t,\Ic+x,a))=\aver{x}{\F_q}{q} \Phi(t,\Ic+x,a)=\frac{1}{|\F_\lf|}+O(\varepsilon(q,G,\Ic)),\]
  where
  \[\varepsilon(q,G,\Ic)=G_\Ic(\alpha(G),|\F_\lf|)+\frac{M_\Ic E(G,M_\Ic, \F_\lf)}{q^{1/2}}.\]
\end{proposition}

As a corollary, we obtain as well a preliminary version of Proposition \ref{prop:removeShifts} about unshifted ``complete'' families, by exchanging summations (see \eqref{eq:exchangeSummation}):

\begin{corollary}\label{cor:familiesShortSumsEUnshifted}
  In the setting of Proposition \ref{prop:familiesShortSums}, assume that for all $a\in\F_\lf$, $\Phi(t,\Ic_k,a)$ does not depend on $k$. Then
  \[\Phi(t,\Ic_k,a)=\frac{|\{x\in\F_q : S(t,\Ic_k'(x))\equiv a\}|}{q}=\frac{1}{|\F_\lf|}+O(\varepsilon(q,G,\Ic)).\]
\end{corollary}

\begin{example}\label{ex:familiesShortSumsEUnshifted}
  In particular, for the family $\Ic$ of Example \ref{ex:familiesShortSums} \ref{ex:familiesShortSumsShifts}, we have for all $k\in\Ic$ that
  \[\{\Ic(k)+x : x\in\F_q\}=\{E+y+x : x\in \F_q\}=\{E+x : x\in\F_q\},\]
  so for all $a\in\F_\lf$ the density $\Phi(t,\Ic_k,a)$ does not depend on $k$. By choosing $\Ic$ as an ``averaging set'' of appropriate size, we would obtain a preliminary version of Proposition \ref{prop:EDfixedSubsets}.
\end{example}

\subsection{Approximate variance}\label{subsec:approximateVariance}

As in \cite{LamzModm}, we now consider the ``approximate variance''
\begin{equation*}
  \left(\Phi(\Ic,a)-\frac{1}{|\F_\lf|}\right)^2,
\end{equation*}
in the sense that we replace the true expected value of the random variable $\Phi(\Ic,a)$ by the approximation given by Proposition \ref{prop:PSIc}. This corresponds to the quantity
\[\left(\Phi(t,\Ic,a)-\frac{1}{|\F_\lf|}\right)^2,\]
and it is clear that bounding the latter gives a result about the distribution of $\Phi(t,\Ic,a)$, uniformly with respect to $a\in\F_\lf$.

\subsubsection{Computation in the model}
\begin{proposition}\label{prop:approxvarianceModel}
  In the notations of Section \ref{subsec:distrShortSumsProbModel}, we have
  \begin{equation*}
    \sum_{a\in\F_\lf}\E \left(\left(\Phi(\Ic,a)-\frac{1}{|\F_\lf|}\right)^2\right)=\frac{1}{|\Ic|}\bigg(1+O \Big(H_\Ic(\alpha(G),|\F_\lf|)\Big)\bigg).
  \end{equation*}
\end{proposition}
\begin{proof}
  As in Proposition \ref{prop:probModel}, we have by orthogonality that
  \begin{eqnarray*}
    \left(\Phi(\Ic,a)-\frac{1}{|\F_\lf|}\right)^2&=&\left|\averC{k}{\Ic}\frac{1}{|\F_\lf|}\sum_{0\neq \psi\in\widehat\F_\lf} \psi \left(S(\Ic(k))-a\right)\right|^2\\
  &=&\frac{1}{|\Ic|^2|\F_\lf|^2}\left|\sum_{0\neq \psi\in\widehat\F_\lf} \psi(-a)\sum_{k\in\Ic}\psi \left(S(\Ic(k))\right)\right|^2\\
  &=&\frac{1}{|\Ic|^2|\F_\lf|^2}\sum_{0\neq \psi_1,\psi_2\in\widehat\F_\lf} \psi_1(-a)\overline{\psi_2(-a)}\\
  &&\times\sum_{k_1,k_2\in\Ic}\psi_1\left(S(\Ic(k_1))\right)\overline{\psi_2 \left(S(\Ic(k_2))\right)}.
\end{eqnarray*}

Again by orthogonality, $\sum_{a\in\F_\lf}\left(\Phi(\Ic,a)-|\F_\lf|^{-1}\right)^2$ is equal to
\[\frac{1}{|\Ic|}\left(\frac{|\F_\lf|-1}{|\F_\lf|}+\frac{1}{|\Ic||\F_\lf|}\sum_{0\neq \psi\in\widehat\F_\lf}\sum_{\substack{k_1,k_2\in\Ic\\k_1\neq k_2}}\frac{\psi(S(\Ic(k_1)))}{\psi(S(\Ic(k_2)))}\right).\]
Since
\[S(\Ic(k_1))-S(\Ic(k_2))=S(\Ic(k_1)\backslash\Ic(k_2))-S(\Ic(k_2)\backslash\Ic(k_1))\] with $\Ic(k_1)\backslash\Ic(k_2)$ and $\Ic(k_2)\backslash\Ic(k_1)$ disjoint, we have by independence
\begin{equation*}
\sum_{\substack{k_1,k_2\in\Ic\\k_1\neq k_2}}\E \bigg[\psi \Big(S(\Ic(k_1))-S(\Ic(k_2))\Big)\bigg]=\sum_{\substack{k_1,k_2\in\Ic\\k_1\neq k_2}}\frac{\E(\psi(Z))^{|\Ic(k_1)\backslash \Ic(k_2)|}}{\E(\psi(Z))^{|\Ic(k_2)\backslash\Ic(k_1)|}}.
\end{equation*}
By the bound on Gaussian sums \eqref{eq:boundGaussSum}, $\E\left(\psi(Z)\right),\E\left(\psi(-Z)\right)\ll|\F_\lf|^{-\alpha(G)}$ uniformly with respect to $\psi$, whence the result.
\end{proof}

\subsubsection{Conclusion} Theorem \ref{thm:familiesShortSums} then follows immediately from Proposition \ref{prop:modelFamiliesSums} and
Proposition \ref{prop:approxvarianceModel}.

\subsection{Estimate and analysis of the error term}\label{subsec:analysisErrorTerm}

We now estimate and analyze the error term
\begin{equation}
  \label{eq:errorTermV}
 V(t,\Ic)\ll\frac{1}{|\Ic|}+\frac{H_\Ic(\alpha,|\F_\lf|)}{|\Ic|}+\frac{H_\Ic(\alpha,|\F_\lf|)M_\Ic}{|\Ic|}\frac{E(G,M_\Ic,|\F_\lf|)}{q^{1/2}}
\end{equation}
in Theorem \ref{cor:familiesShortSums}, where $\alpha=\alpha(G)$.

\subsubsection{Estimates for $V(t,\Ic)$}

\begin{definition}\label{def:constantsIc3}
  For a family $\Ic\to\Pc(\F_q)$, we define $m_\Ic=\max_{k\in\Ic} |\Ic(k)|$ and $A_\Ic=\min_{k_1\neq k_2\in\Ic} |\Ic(k_1)\Delta\Ic(k_2)|$.
\end{definition}

\begin{lemma}\label{lemma:estimatesHi}
  In the notations of Definitions \ref{def:constantsIc}, \ref{def:constantsIc2} and \ref{def:constantsIc3}, we have the bounds $M_\Ic\le|\Ic|m_\Ic$, $1\le A_\Ic\le 2M_\Ic$, and
  \[H_\Ic(\alpha,n)\ll \frac{\max(h_\Ic(d) : 1\le d\le 2m_\Ic)}{|\Ic|n^{\alpha A_\Ic}}\le \frac{|\Ic|}{n^{\alpha A_\Ic}}.\]
  The bound for $H_\Ic(\alpha,n)$ can be improved for the following families:
  \begin{enumerate}
  \item\label{prop:estimatesHiOrdered} If $\Ic$ is totally ordered by some order $<$ with $\Ic(k_1)\subset\Ic(k_2)$ for $k_1<k_2$, and if $\Ic$ is determined by its cardinality, then $H_\Ic(\alpha,n)\ll n^{-\alpha A_\Ic}$. In particular, this holds for the family $\Ic\subset\{1,\dots,p\}$, $\Ic(k)=\{1,\dots, k\}^e\subset\F_p^e\cong\F_q$ of Example \ref{ex:familiesShortSums} \ref{ex:familiesShortSumsIntervals}.
  \item \label{prop:estimatesHiShifts} For the family $\Ic$ of Example \ref{ex:familiesShortSums} \ref{ex:familiesShortSumsShifts}, we have
    \[H_\Ic(\alpha,n)\ll\frac{\max_{0\le d<|E|} \big|\{y\in B_\Ic : |E\cap(E+y)|=d\}\big|}{n^{\alpha A_\Ic}}\]
    where $B_\Ic=\{y_1-y_2 : y_1,y_2\in\Ic\text{ distinct}\}$. In particular, if
    \begin{equation}
      \label{eq:IcsubsetRed}
      \Ic\subset\prod_{i=1}^e [0,p-\max_{x\in E} x_i],
    \end{equation}
    then $H_\Ic(\alpha,n)\ll \frac{|B_E|}{n^{\alpha A_\Ic}}$. This is an improvement over the previous bound if $|\Ic|>|E|$.
  \end{enumerate}
\end{lemma}
\begin{proof} The trivial bound $1\le h_\Ic(d)\le |\Ic|^2$ gives the first bound for $H_\Ic(\alpha)$.
  \ifamsart
  \begin{enumerate}[leftmargin=*]
  \else
  \begin{enumerate}
  \fi
\item Under the first hypothesis,
  \[h_\Ic(d)=2|\{k_1<k_2 : |\Ic(k_2)|=|\Ic(k_1)|+d \}|.\]
  for all $d\ge 1$. If $\Ic(k)$ is moreover is determined by its cardinality, then $h_\Ic(d)\le 2|\Ic|$.
\item We have
  \begin{eqnarray*}
    h_\Ic(d)&=&\big|\{y_1, y_2\in \Ic \text{ distinct }: |E\cap (E+(y_2-y_1))|=|E|-d/2\}\big|\\
    &\ll&|\Ic|\cdot\big|\{y\in B_\Ic : |E\cap(E+y)|=|E|-d/2\}\big|,
  \end{eqnarray*}
  whence the first statement. If $y\in\prod_{i=1}^e[0,p-\max_{x\in E} x_i]^e\subset\F_p^e\cong\F_q$ (to avoid reductions modulo $p$), then $B_E\cap (B_E+y)=\varnothing$ if $y\not\in B_E$, which gives the second assertion.
\end{enumerate}
\end{proof}


\subsubsection{Analysis of the parameters}

The next lemma provides a general analysis of the error term \eqref{eq:errorTermV} that we will use to handle the various examples of Theorem \ref{thm:familiesShortSums}.

\begin{lemma}\label{lemma:analysisParameters}
  We have $V(t,\Ic)=o(1)$ if the following three conditions hold:
  \begin{enumerate}
  \item $|\Ic|\to+\infty$.
  \item\label{item:analysisParameters2} $H_\Ic:=H_\Ic(\alpha(G),|\F_\lf|)=o(|\Ic|)$.
  \item The sum
    \begin{equation}
      \label{eq:sumMIcParameters}
      M_\Ic+\frac{2\left(\log(M_\Ic/|\Ic|)+\log{H_\Ic}\right)}{\log(|G||G^\sharp|)}
    \end{equation}
    is strictly smaller than
      \[\begin{cases}
    \frac{1}{\beta_+(G)}\frac{\log{q}}{\log|\F_\lf|}-\frac{2\beta_-(G)}{\beta_+(G)}&\text{if } G\text{ classical}\\
    \frac{\log{q}}{\log{d}}-1&\text{if } G\text{ cyclic}.
  \end{cases}\]
  \end{enumerate}
  If we have
  \begin{equation}
    \label{eq:MIclogHIc}
    M_\Ic=|\Ic|\text{ and }\log{H_\Ic}\ll \log(|G||G^\sharp|),
  \end{equation}
  this implies that
  \[
  \begin{cases}
    \log|\F_\lf|=o(\log{q})&\text{if } G\text{ classical}\\
    \log{d}=o(\log{q})&\text{if } G\text{ cyclic}.
  \end{cases}
  \]
\end{lemma}
\begin{remarks}~
  \begin{enumerate}
  \item By Lemma \ref{lemma:estimatesHi}, $H_\Ic/|\Ic|\le |\F_\lf|^{-\alpha A_\Ic}$, so Condition \ref{item:analysisParameters2} holds if $|\F_\lf|\to+\infty$ or if $H_\Ic=O(1)$ (e.g. for a family satisfying Lemma \ref{lemma:estimatesHi} \ref{prop:estimatesHiOrdered}).
  \item If $\log|\F_\lf|=o(\log{q})$ and $p<|\F_\lf|$, note that we must take $e\to+\infty$ (see Section \ref{subsec:limitationRange}).
  \end{enumerate}  
\end{remarks}

\begin{remark}\label{rem:optimalSizeI}
  The optimal size for $M_\Ic$ is therefore $M_\Ic\approx \frac{2\varepsilon\log{q}-\log(|G|/|G^\sharp|)}{\log(|G||G^\sharp|)}$ for some $\varepsilon\in(0,1/2)$, giving
  \[V(t,\Ic)\ll_\varepsilon \Big(m_\Ic+H_\Ic(\alpha(G),|\F_\lf|)\Big)\frac{\log(|G||G^\sharp|)}{\log{q}}+\frac{1}{q^{1/2-\varepsilon}}.\]
\end{remark}

\subsection{Removing the shifts}\label{subsec:removeShifts}

The general setting to obtain asymptotic equidistribution for unshifted ``complete'' families from Theorem \ref{thm:familiesShortSums}. is the following:

\begin{proposition}\label{prop:removeShifts}
  Under the hypotheses of Theorem \ref{thm:familiesShortSums}, assume furthermore that:
  \ifamsart
  \begin{enumerate}[leftmargin=*]
  \else
  \begin{enumerate}
  \fi
  \item \label{item:IcIc'f1f2} For some family $\Ic_2$ and functions $f_1: \F_q\times\Ic\to\Ic_2$, $f_2: \F_q\to\F_\lf$ we have
    \[S(t,\Ic'(k,x))=S\big(t,\Ic_2(f_1(x,k))\big)+f_2(x) \hspace{0.2cm} (k\in\Ic, x\in\F_q).\]
  \item \label{item:kinvariantUptoError}
    There exists a function $f_3: \Ic\to\R_+$ and a family $\Ic_3: \F_q\to\Pc(\F_q)$ such that for all $a\in\F_\lf$ and $k\in\Ic$
    \begin{eqnarray*}
      |\{x\in\F_q : S(t,\Ic_2(f_1(x,k)))\equiv a\}|&=&|\{x\in\F_q : S(t,\Ic_3(x))\equiv a\}|\\
      &&+ O(f_3(k)).
    \end{eqnarray*}
    In particular, if the set $\{\Ic_2(f_1(x,k)) : x\in\F_q\}$ does not depend on $k\in\Ic$, this holds true with $f_3=0$ and $\Ic_3(x)=\Ic_2(f_1(x,k_0))$ for any $k_0\in\Ic$.
  \end{enumerate}
  Then, if $||f_3||_\infty/q\le 1$,
  \[\Phi(t,\Ic_3,a)=\frac{1}{|\F_\lf|}+O\left(V(t,\Ic)^{1/2}+\left(\frac{|\F_\lf|||f_3||_\infty}{q}\right)^{1/2}\right)\]
  uniformly with respect to $a\in\F_\lf$.
\end{proposition}

In other words, we use $\Ic$ as an ``averaging family'' to get asymptotic equidistribution for the complete family $\Ic_3$, and the error term depends on $\Ic$. Note that the averaging over $a\in\F_\lf$ gives some additional freedom in comparison with the preliminary version from Corollary \ref{cor:familiesShortSumsEUnshifted}.

\begin{proof}
  Under Condition \ref{item:IcIc'f1f2}, Theorem \ref{thm:familiesShortSums} gives
\[\sum_{a\in\F_\lf}\frac{1}{q}\sum_{x\in\F_q} \left(\frac{|\{k\in\Ic : S(t,\Ic_2(f_1(x,k)))\equiv a\}|}{|\Ic|}-\frac{1}{|\F_\lf|}\right)^2\ll V(t,\Ic)\]
by exchanging the summation over $a$ and $x$ and exploiting the averaging over $a$. By the Cauchy-Schwarz inequality,
\[\sum_{a\in\F_\lf} \left( \aver{x}{\F_q}{q} \frac{|\{k\in\Ic : S(t,\Ic_2(f_1(x,k)))\equiv a\}|}{|\Ic|}-\frac{1}{|\F_\lf|}\right)^2\ll V(t,\Ic)\]
By exchanging the summations over $k$ and $x$, this is equal to
\[\sum_{a\in\F_\lf} \left(\averC{k}{\Ic} \frac{|\{x\in\F_q : S(t,\Ic_2(f_1(x,k)))\equiv a\}|}{q}-\frac{1}{|\F_\lf|}\right)^2.\]
Finally, by Condition \ref{item:kinvariantUptoError},
\[\sum_{a\in\F_\lf} \left(\frac{|\{x\in\F_q : S(t,\Ic_3(x))\equiv a\}|}{q}-\frac{1}{|\F_\lf|}+O \left(\frac{||f_3||_\infty}{q}\right) \right)^2\ll V(t,\Ic).\]
\end{proof}

\begin{example}\label{ex:familiesShortSumsEUnshifted2}
  For the family $\Ic$ of Example \ref{ex:familiesShortSums} \ref{ex:familiesShortSumsShifts}, we have by Example \ref{ex:familiesShortSumsEUnshifted} that:
  \begin{itemize}
  \item Condition \ref{item:IcIc'f1f2} of Proposition \ref{prop:removeShifts} holds with $\Ic_2=\Ic$, $f_1(x,k)=k+x$ and $f_2=0$.
  \item Condition \ref{item:kinvariantUptoError} holds with $f_3=0$ and $\Ic_3=\Ic_0=\Ic'(0,\cdot)$, since $\{ x+k : x\in\F_q\}=\F_q$ for all $k\in\F_q$.
  \end{itemize}
\end{example}

\subsection{Applications of Proposition \ref{prop:removeShifts}}\label{subsec:applicationRemoveShifts}

In the following paragraphs, we use Proposition \ref{prop:removeShifts} to prove Propositions \ref{prop:EDfixedSubsets}, \ref{prop:partialIntervals} and \ref{prop:partialIntervalShifts}.

The general idea is to find an averaging family $\Ic$ of size large enough to get asymptotic equidistribution in Corollary \ref{cor:familiesShortSums}, and the assumptions we make are precisely to allow that, according to Lemma \ref{lemma:analysisParameters}.

\subsubsection{Choice of the averaging family}

When $e>1$, we will have $\Ic=\Ic_1\times\dots\times\Ic_e$ with $\Ic_i$ of determined structure and whose size can be chosen freely in some range. Since the final bound depends only on the size of $\Ic$, we need to choose the sizes of the $\Ic_i$ to attain the optimal/desired size for $\Ic$. Note however that in the case $|\Ic|\le \frac{\log{q}}{\log{|\F_\lf|}}$ and $p<|\F_\lf|$ (see Section \ref{subsec:limitationRange}), we have
\[|\Ic|^{1/e}\le\left(\frac{\log{q}}{\log{|\F_\lf|}}\right)^{1/e}=e^{1/e}\left(\frac{\log{p}}{\log{|\F_\lf|}}\right)^{1/e}\le 1,\]
which shows that the choice $|\Ic_1|=\dots=|\Ic_e|\approx |\Ic|^{1/e}$ is impossible. More carefully, we take $|\Ic_1|=\dots=|\Ic_a|\approx|\Ic|^{1/a}$ with $1\le a<e$ of optimal size given by:

\begin{lemma}\label{lemma:choiceIi}
  Let $I\ge 1$, $p\ge 2$, $e\ge 2$ be integers, and let $0<\delta\le 1$. If $\log{I}\le (e-1)\log(\delta p)$, there exist integers $I_1\in\{1,\dots, \delta p\}$ and $1\le a\le e$ such that $I_1^a=I \left(1+o(1)\right)$ for $I$ large enough.
\end{lemma}
\begin{proof}
  It suffices to take $I_1=\floor{I^{1/a}}$ with $a=\ceil{\log{I}/\log(\delta p)}\ge 1$ so that $I_1\in\{1,\dots, \delta p\}$, $a=o(I^{1/a})$,
  \[I_1^a=I+O\left(aI^{1-1/a}\right)=I \left(1+O(aI^{-1/a})\right)=I \left(1+o(1)\right),\]
  and the condition $a\le e$ holds if $\log{I}\le (e-1)\log(\delta p)$
\end{proof}
\begin{example}
  The condition $\log{I}\le (e-1)\log(\delta p)$ is satisfied if $I\le \log{q}=e\log{p}$ as in Lemma \ref{lemma:analysisParameters}, up to taking $p$ large enough if $\delta<1$.
\end{example}

\subsubsection{Shifts of subsets}

We first consider the family of Example \ref{ex:familiesShortSums} \ref{ex:familiesShortSumsShifts}: for $\Ic,E\subset\F_q$, we let $\Ic(k)=E+k$ ($k\in\Ic$).

\begin{proof}[Proof of Proposition \ref{prop:EDfixedSubsets}]
  By Example \ref{ex:familiesShortSumsEUnshifted2}, Proposition \ref{prop:removeShifts} can be applied.

  By Lemma \ref{lemma:estimatesHi}, since $m_\Ic=|E|$, the sum \eqref{eq:sumMIcParameters} is
\[\le |\Ic||E|+\frac{2(\log{|E|}+\log{|B_E|}-\alpha(G) A_\Ic|\F_\lf|)}{\log(|G||G^\sharp|)}\]
if $\Ic\subset\prod_{i=1}^e[1,p-\max_{x\in E} x_i]$. By Lemma \ref{lemma:analysisParameters}, we want that for some $\varepsilon\in(0,1/2)$,
\[|\Ic|<\frac{1}{|E|}\left(\frac{2\varepsilon\log{q}-2\log{|E|}-2\log{|B_E|}}{\beta_+(G)\log|\F_\lf|}-\frac{\beta_-(G)}{\beta_+(G)}+2\alpha(G)A_\Ic\right)\]
if $G$ is classical, and
\[|\Ic|<\frac{1}{|E|}\left(\frac{2\varepsilon\log{q}-2\log{|E|}-2\log{|B_E|}}{\log{d}}+2\alpha(G)A_\Ic-1\right)\]
if $G=\mu_d$.

When the sheaf is of the form $\Lc_{\chi(f)}$ with $f\neq X$, we impose that $\Ic\subset\prod_{i=1}[1,p/\deg(f_1)-\max_{x\in E} x_i)$, so that it is $\bigcup_{k\in\Ic} \Ic(k)$-compatible by Example \ref{ex:noZeromSum}.

Under the assumptions of Proposition \ref{prop:EDfixedSubsets}, we can choose $\Ic$ as large as possible satisfying the above conditions by Lemma \ref{lemma:choiceIi}.
\end{proof}

\subsubsection{Intervals}

Let us now analyze the family $\Ic\subset\{1,\dots,p\}$, $k\mapsto \{1,\dots, k\}$ of Example \ref{ex:familiesShortSums} \ref{ex:familiesShortSumsIntervals}. For the Legendre symbol, this is the case of \cite{LamzModm}.

\begin{lemma}\label{lemma:fxk}
  For any $f: \N\to\C$ and $k\in\N$,
  \[\sum_{x=1}^p f(x+k)=\sum_{x=1}^p f(x)+O(k||f||_\infty),\]
  and the error term can be removed if $f$ is $p$-periodic.
\end{lemma}
\begin{proof}
  It suffices to write
  \begin{eqnarray*}
    \sum_{x=1}^p f(x+k)&=&\sum_{x=1+k}^{p+k} f(x)=\left(\sum_{x=1}^p-\sum_{x=1}^{k}+\sum_{x=p+1}^{p+k}\right) f(x)\\
    &=&\sum_{x=1}^p f(x)+O(k||f||_\infty).
  \end{eqnarray*}
\end{proof}

\begin{example}
  For $f(x)=\delta_{S(t,\{1,\dots, x\})\equiv a}$, we have
\[f(x+p)=\delta_{S(t,\F_p)+S(t,\{p+1,\dots, x+p\})\equiv a}=\delta_{S(t,\{1,\dots,p\})+S(t,\{1,\dots,x\})\equiv a},\]
and $f$ is $p$-periodic if
\begin{equation}
  \label{eq:S1p0}
  S(t,\{1,\dots,p\})=0
\end{equation}
(i.e. orthogonality with constant functions).
\end{example}

\begin{proof}[Proof of Proposition \ref{prop:partialIntervals}]
  We apply again Proposition \ref{prop:removeShifts}:

  Condition \ref{item:IcIc'f1f2} holds with $\Ic_2=\Ic$, $f_1(x,k)=k+x$ and $f_2(x)=S(t,\Ic(x))$ since
\[S(t,\{1+x,\dots, k+x\})=S(t,\{1,\dots, k+x\})-S(t,\{1,\dots, x\})\]
for $k\in\Ic$, $x\in\F_p$.

By Lemma \ref{lemma:fxk}, Condition \ref{item:kinvariantUptoError} holds with $||f_3||_\infty\le \max_{k\in\Ic} k$, and with no error term if the trace function considered satisfies \eqref{eq:S1p0}. Otherwise, we add the error term
\[\left(\frac{|\F_\lf||\Ic|}{p}\right)^{1/2}\ll \left(\frac{|\F_\lf|\log{p}}{p\log{d}}\right)^{1/2}.\]

If the sheaf is a Kummer sheaf $\Lc_{\chi(f)}$ we impose $\max_{k\in\Ic}k <p/\deg(f_1)$, so that it is $\bigcup_{k\in\Ic}\Ic(k)$-compatible by Example \ref{ex:noZeromSum}.\\

We may then choose $\Ic$ as large as permitted by Lemma \ref{lemma:analysisParameters}, i.e. $|\Ic|\approx \frac{\log{p}}{\log{d}}$, noting that for Kummer sheaves as above, we have $\frac{p}{\deg(f_1)}>\frac{\log{p}}{\log{d}}$ for $p$ large enough ($\deg(f_1)$ being bounded independently from $q$).

\end{proof}

\begin{remark}
  For Kloosterman sums, we have seen in Section \ref{subsec:limitationRange} that it is necessary to take $e\to+\infty$ so that $V(t,\Ic)=o(1)$. Hence, Proposition \ref{prop:partialIntervals} does not apply to them. Unfortunately, issues arise when we try to generalize the proposition to $e>1$. Indeed, for the family
  \[\Ic\subset\{1,\dots,p\}^e, \ \Ic(k)=\prod_{i=1}^e \{1,\dots, k_i\}, \ k=(k_1,\dots,k_e)\in\Ic,\]
  of Example \ref{ex:familiesShortSums} \ref{ex:familiesShortSumsBoxes}, we have $\Ic'(k,x)=\prod_{i=1}^e \{1+x_i,\dots, x_i+k_i\}$ for all $x=(x_1,\dots,x_e)\in\F_q$. As above, we can decompose $\{1+x_i\dots, x_i+k_i\}=\{1,\dots, x_i+k_i\}\backslash\{1,\dots, x_i\}$ and write
\[S(t,\Ic'(k,x))=\sum_{a_1,\dots,a_e\in\{0,1\}} (-1)^{\sum_{i=1}^e (a_i+1)} S\big(t,\Ic((x_i+a_ik_i)_i)\big).\]
However, there are now ``diagonal terms'' including $x_i$ and $x_j+k_j$ ($i\neq j$), preventing us from applying Proposition \ref{prop:removeShifts} with $f_1(x,k)=x+k$ and $f_3=0$ as before. On the other hand, using Lemma \ref{lemma:fxk} would give a large error $||f_3||_\infty\approx ep^{e-1}$ because small intervals of size $k_i$ combine with large intervals of size $p-k_j$ into large ``diagonal'' terms. This would give an error term $|\F_\lf|ep^{e-1}/q=|\F_\lf|e/p>e$ in the final expression for the density, which is not acceptable when $e\to+\infty$. These diagonal terms compensate each other if complete sums in one parameter of the form $S(t,E_1\times\dots\times E_i\times\F_p\times E_{i+2}\times\dots\times E_e)$ vanish, for $E_i\subset\F_p$. Being defined as Fourier transforms of functions vanishing at $0$, Kloosterman sums verify $S(\Kl_{n,q},\{1,\dots,p\}^e)=0$, but the former sums do not vanish in general.
\end{remark}

\subsubsection{Small intervals with shifts of subsets} We finally consider Proposition \ref{prop:partialIntervalShifts}, which is about a family of type of Example \ref{ex:familiesShortSums} \ref{ex:familiesShortSumsCombining} and gives a variant of Proposition \ref{prop:partialIntervals} for $e>1$ (in particular for Kloosterman sums).

\begin{proof}[Proof of Proposition \ref{prop:partialIntervalShifts}]
  Let us write $\Ic=\Ic_1\times\Ic_2\subset\F_p\times\F_p^{e-1}$ and let $E=E_2\times\dots\times E_e$. Then
\begin{eqnarray*}
  M_\Ic&=&\Big|\bigcup_{(k_1,k_2)\in\Ic} \{1,\dots, k_1\}\times (E+k_2)\Big|\\
  &\le& \Big|\bigcup_{k_1\in\Ic_1}\{1,\dots, k_1\}\Big|\times|\Ic_2||E|\le |\Ic||E|
\end{eqnarray*}
and for any $\varepsilon>0$ and $d\ge 1$, we have
\begin{eqnarray*}
  h_\Ic(d)&\le&\sum_{k_1,k_1'\in\Ic_1} \Big|\{k_2,k_2'\in\Ic_2 : |k_1-k_1'||E\Delta(E+(k_2-k_2'))|=d\}\Big|\\
  &=&|\Ic_1|\sum_{1\le d'\mid d} \Big|\{k_2,k_2'\in\Ic_2 : |E\Delta(E+(k_2-k_2'))|=d/d'\}\Big|\\
  &\le&|\Ic_1||B_E||\Ic_2|\tau(d)\ll_\varepsilon(|\Ic||B_E|)^{1+\varepsilon}
\end{eqnarray*}
if $\Ic_2\subset\prod_{i=2}^e [1,p-\max_{x\in E} x_i]$, by Lemma \ref{lemma:estimatesHi}. As for Propositions \ref{prop:EDfixedSubsets} and \ref{prop:partialIntervals}, if the sheaf is a Kummer sheaf $\Lc_{\chi(f)}$, we impose $\max_{k\in\Ic_1}k <p/\deg(f_1)$ and $\Ic_2\subset[1,p/\deg(f_1))$, to ensure $\bigcup_{k\in\Ic} \Ic(k)$-compatibility.\\

The conclusion then follows by using Lemmas \ref{lemma:analysisParameters} and \ref{lemma:choiceIi} as in Proposition \ref{prop:EDfixedSubsets}.

\end{proof}

\bibliographystyle{alpha}
\ifamsart
\small
\bibliography{references}
\normalsize
\else
\bibliography{references}
\fi

\end{document}